%% file: puzzle2eq.tex
\documentclass{amsart}
\usepackage{amsmath, amsfonts, amssymb}
\usepackage{float}
\usepackage{graphicx}
\usepackage{psfrag}
\usepackage{tableau}

\input{preamble}
\newcommand{\pic}[2]{\includegraphics[scale=#1]{#2}}

\begin{document}

\title[Mutations of puzzles and equivariant cohomology]%
{Mutations of puzzles and equivariant cohomology of two-step flag
  varieties}

\date{August 27, 2014}

\author{Anders~Skovsted~Buch}
\address{Department of Mathematics, Rutgers University, 110
  Frelinghuysen Road, Piscataway, NJ 08854, USA}
\email{asbuch@math.rutgers.edu}


\thanks{The author was supported in part by NSF grants DMS-0906148 and
  DMS-1205351.}

\begin{abstract}
  We introduce a mutation algorithm for puzzles that is a
  three-direction analogue of the classical jeu de taquin algorithm
  for semistandard tableaux.  We apply this algorithm to prove our
  conjectured puzzle formula for the equivariant Schubert structure
  constants of two-step flag varieties.  This formula gives an
  expression for the structure constants that is positive in the sense
  of Graham.  Thanks to the equivariant version of the `quantum equals
  classical' result, our formula specializes to a
  Littlewood-Richardson rule for the equivariant quantum cohomology of
  Grassmannians.
\end{abstract}

\maketitle

\input{intro}
\input{equiv}
\input{recurse}
\input{mutate}

\input{aura}

%\bibliography{../BibTeX/database.bib,puzzle2eq}
%\bibliographystyle{amsplain}

\input{bibliography}

\end{document}

%% file: preamble.tex
\newtheorem{thm}{Theorem}[section]
\newtheorem{lemma}[thm]{Lemma}
\newtheorem{prop}[thm]{Proposition}
\newtheorem{cor}[thm]{Corollary}

\theoremstyle{definition}
\newtheorem{remark}[thm]{Remark}
\newtheorem{example}[thm]{Example}
\newtheorem{defn}[thm]{Definition}

\numberwithin{table}{subsection}

\DeclareMathOperator{\Gr}{Gr}
\DeclareMathOperator{\Fl}{Fl}

\DeclareMathOperator{\GL}{GL}

\DeclareMathOperator{\Ker}{Ker}
\DeclareMathOperator{\Span}{Span}

\DeclareMathOperator{\QH}{QH}

\DeclareMathOperator{\codim}{codim}

\DeclareMathOperator{\weight}{wt}
\DeclareMathOperator{\scabs}{scabs}
\newcommand{\Bo}{{\mathbf B}}
\newcommand{\ssm}{\smallsetminus}
\renewcommand{\Re}{\operatorname{Re}}

\newcommand{\N}{{\mathbb N}}
\newcommand{\Z}{{\mathbb Z}}

\newcommand{\R}{{\mathbb R}}
\newcommand{\C}{{\mathbb C}}
\newcommand{\cA}{{\mathcal A}}
\newcommand{\cB}{{\mathcal B}}
\newcommand{\cD}{{\mathcal D}}

\newcommand{\pt}{\text{pt}}

\newcommand{\ev}{\operatorname{ev}}
\newcommand{\wtil}{\widetilde}
\newcommand{\wh}{\widehat}
\newcommand{\wb}{\overline}
\newcommand{\hmm}[1]{\mbox{}\hspace{#1mm}\mbox{}}

\newcommand{\ignore}[1]{}

\newcommand{\Mb}{\wb{\mathcal M}}

\newcommand{\xra}{\xrightarrow}

\newcommand{\border}{\text{\large$\triangle$}}
\newcommand{\temp}{\text{\rm temp}}
\newcommand{\gash}{\text{\rm gash}}
\newcommand{\scab}{\text{\rm scab}}
\def\ds{\displaystyle}

\def\la{\lambda}
\def\La{\Lambda}

\def\Gright{\overrightarrow{\mathcal G}}

\def\Rright{\overrightarrow{\mathcal R}}
\def\Rleft{\overleftarrow{\mathcal R}}
\def\Pright{\overrightarrow{\mathcal P}}
\def\Pleft{\overleftarrow{\mathcal P}}

%% file: intro.tex
\section{Introduction}\label{sec:intro}

In 1999 Allen Knutson circulated a conjecture stating that any
Schubert structure constant of the cohomology ring of a partial flag
variety $X = \GL(n)/P$ can be expressed as the number of puzzles that
can be created using a list of triangular puzzle pieces with matching
side labels \cite{knutson:conjectural}.  While this conjecture was
proved in the special case where $X$ is a Grassmann variety
\cite{knutson.tao.ea:honeycomb, knutson.tao:puzzles}, Knutson
discovered counter examples to his general conjecture.  In later work
by Kresch, Tamvakis, and the author
\cite{buch.kresch.ea:gromov-witten} it was proved that the (3 point,
genus zero) Gromov-Witten invariants of Grassmannians are special
cases of the Schubert structure constants of two-step flag varieties
$\Fl(a,b;n)$.  In fact, the map that sends a rational curve to its
kernel-span pair \cite{buch:quantum} provides a bijection between the
curves counted by a Gromov-Witten invariant and the points of
intersection of three general Schubert varieties in a two-step flag
variety.  Supported by computer verification, it was suggested in
\cite{buch.kresch.ea:gromov-witten} that Knutson's conjecture might
correctly predict the Schubert structure constants of two-step flag
varieties.  This case of the conjecture has recently been proved
\cite{buch.kresch.ea:puzzle}.  A different positive combinatorial
formula for the cohomological structure constants of two-step flag
varieties had earlier been proved by Co\c{s}kun
\cite{coskun:littlewood-richardson}.  See also
\cite{knutson.purbhoo:product} for a relation between puzzles and
Belkale-Kumar coefficients.

The cohomology ring of a homogeneous space $X = G/P$ generalizes to
the equivariant cohomology ring $H^*_T(X;\Z)$, whose structure
incorporates the action of a torus $T$.  The Schubert structure
constants of this ring are elements of $H^*_T(\pt,\Z)$, which can be
identified with the polynomial ring $\Z[y_1,\dots,y_n]$.  Graham has
proved that the equivariant Schubert structure constants are
polynomials with positive coefficients in the differences
$y_{i+1}-y_i$ \cite{graham:positivity}.  Knutson and Tao's paper
\cite{knutson.tao:puzzles} proves an equivariant generalization of the
puzzle rule for Grassmannians that makes Graham's positivity result
explicit.  The equivariant puzzles in this rule are composed by
triangular puzzle pieces as well as rhombus shaped {\em equivariant
  puzzle pieces}.  The equivariant pieces are required to be vertical,
and each equivariant piece contributes a {\em weight\/} $y_j-y_i$ with
$i<j$, where the values of $i$ and $j$ depend on the location of the
equivariant piece in the puzzle.  Knutson and Tao define the weight of
an equivariant puzzle to be the product of the weights of its
equivariant pieces, and prove that any equivariant Schubert structure
constant of a Grassmann variety is equal to the sum of the weights of
a collection of equivariant puzzles.  A different Graham-positive
formula for the equivariant structure constants of Grassmannians was
later obtained independently by Molev
\cite{molev:littlewood-richardson*1} and Kreiman
\cite{kreiman:equivariant}.  In addition, Knutson has recently
generalized the puzzle rule for Grassmannians to equivariant
$K$-theory \cite{knutson:puzzles}.

Efforts to prove Knutson's conjecture more than a decade ago resulted
in a conjectured Graham-positive formula for the equivariant Schubert
structure constants of any two-step flag variety $\Fl(a,b;n)$, which
generalizes both the equivariant puzzle rule for Grassmannians and the
cohomological puzzle rule for two-step flag varieties.  This
conjecture was published in Co\c{s}kun and Vakil's survey
\cite{coskun.vakil:geometric} together with a suggested correction of
Knutson's cohomological conjecture for three-step flag varieties.  The
main result in the present paper is a proof of the conjectured
equivariant puzzle formula for two-step flag varieties
(Theorem~\ref{thm:puzzle}).

Our paper \cite{buch.mihalcea:quantum} with Mihalcea proves that the
equivariant Gromov-Witten invariants of Grassmannians are special
cases of the equivariant Schubert structure constants of two-step flag
varieties, thus generalizing the `quantum equals classical' result
from \cite{buch.kresch.ea:gromov-witten}.  Our puzzle formula
therefore specializes to a Littlewood-Richardson rule for the
equivariant quantum cohomology of Grassmannians that accords with
Mihalcea's result \cite{mihalcea:positivity} that the equivariant
Gromov-Witten invariants satisfy Graham positivity.  While different
formulas for equivariant Gromov-Witten invariants are known
\cite{mihalcea:equivariant, mihalcea:equivariant*1,
  beazley.bertiger.ea:rim-hook}, positive formulas have not been
available earlier for either the equivariant cohomology of two-step
flag varieties or the equivariant quantum cohomology of Grassmannians.

The main combinatorial construction in our paper is an algorithm
called {\em mutation of puzzles}, which is analogous to
Sch\"utzenberger's jeu de taquin algorithm for semistandard Young
tableaux.  Recall that the jeu de taquin algorithm operates on Young
tableaux that contain a {\em flaw\/} in the form of an empty box, and
works by making natural changes to move the empty box to a different
location in the tableau.  Our mutation algorithm similarly operates on
{\em flawed puzzles}.  A flaw in a puzzle can be a {\em pair of
  gashes\/} on the boundary, a {\em marked scab}, or a {\em temporary
  puzzle piece}.  Gashes and scabs are also present in Knutson and
Tao's work \cite{knutson.tao:puzzles}, whereas temporary puzzle pieces
are new in this paper.  Flawed puzzles that contain a gash pair or a
marked scab can be mutated in exactly one way.  On the other hand, a
puzzle containing a temporary puzzle piece has exactly three
mutations, which correspond to moving the temporary piece in three
different directions.  The mutation algorithm therefore organizes the
set of all flawed puzzles into a trivalent graph with leaves (see
Figure~\ref{ssec:mutations}).  In contrast, the jeu de taquin
algorithm offers only two choices for moving an empty box in a
tableau.  Our definition of flawed puzzles allows equivariant puzzle
pieces to appear in arbitrary orientations and also allows the shape
of a puzzle to be a hexagon.  This ensures that rotations of puzzles
are again puzzles, which in turn simplifies the definition of the
mutation algorithm.

The changes made to a puzzle during a mutation are based on the
following observation.  Suppose that a puzzle contains a {\em gash},
i.e.\ the labels of two puzzle pieces next to each other do not match.
Then there is {\em at most one way\/} to replace either of these
pieces with a different puzzle piece such that the gash disappears and
only one new gash is created by the replacement.  This provides a
natural method for moving a gash around in a puzzle, which we call
{\em propagation\/} of the gash.  Given a flawed puzzle, the mutation
algorithm first {\em resolves\/} the flaw by replacing it with two
gashes.  Both of these gashes are then propagated as far as possible.
Our main technical result states that the two gashes will propagate to
the same location in the puzzle, where they create a new flaw.  While
the mutation algorithm itself can be formulated in terms of general
principles, the proof that it works requires some case by case
analysis.  For example, the proof of the above-mentioned technical
result is a winding number argument that is justified with case
checking.  Even so, our construction of the mutation algorithm applies
without change to the puzzles appearing in the conjectured formula for
the cohomology of three-step flag varieties; this will be explained in
\cite{buch:3step} together with the consequences for this conjecture.
It is natural to speculate that a correction of Knutson's general
conjecture for $\GL(n)/P$, if one exists, should be formulated in
terms of puzzles that can be mutated.

Thanks to an idea that originates in Molev and Sagan's work on
products of factorial Schur functions
\cite{molev.sagan:littlewood-richardson}, any formula for the
equivariant Schubert structure constants of a homogeneous space $X =
G/P$ can be proved by verifying certain recursive identities
associated to multiplication with divisor classes, together with
showing that the formula is compatible with restriction of equivariant
Schubert classes to torus-fixed points \cite{kostant.kumar:nil*1,
  andersen.jantzen.ea:representations, billey:kostant}.  This method
was used in \cite{knutson.tao:puzzles}.  Molev and Sagan's method
requires the verification of $2r$ families of recursive identities,
where $r$ is the rank of the Picard group of $X$.  By working with
equivariant cohomology with coefficients in the polynomial ring $R =
\C[\delta_0,\delta_1,\delta_2]$, we combine the 4 families of
identities required for a two-step flag variety into a single
recursive identity that involves powers of a 12-th root of unity
$\zeta \in \C$.  The proof that this identity is satisfied by the
constants defined by our puzzle formula involves assigning an {\em
  aura\/} in the ring $R$ to various objects related to puzzles.  Here
the powers of $\zeta$ are used as unit vectors whose directions
correspond to puzzle angles, and the variables $\delta_0$, $\delta_1$,
and $\delta_2$ correspond to simple puzzle labels.  The recursive
identity then follows from the mutation algorithm together with simple
identities among auras.  Our proof of the puzzle formula is logically
self-contained starting from basic properties of equivariant
cohomology \cite{kostant.kumar:nil*1, anderson.fulton:equivariant} and
the Monk/Chevalley formula \cite{chevalley:decompositions,
  monk:geometry}.

The proofs of the puzzle formulas in \cite{knutson.tao:puzzles,
  buch.kresch.ea:puzzle} rely on bijections of puzzles to establish
certain basic identities.  These bijections are formulated in terms of
{\em propagation rules\/} stating that a small region of a puzzle with
a particular look must be changed in a specified way.  Knutson and
Tao's bijection requires around 10 rules, while the bijection in
\cite{buch.kresch.ea:puzzle} uses a list of 80 propagation rules.  In
contrast the mutation algorithm is defined without lists of rules.  In
Section~\ref{ssec:biject} we explain how mutations of puzzles can be
used to give a new construction of the bijections from
\cite{knutson.tao:puzzles, buch.kresch.ea:puzzle}.  This construction
involves that some areas of a puzzle can be changed by more than one
mutation, which, at least for two-step puzzles, is simpler than giving
a direct description of the end result of the bijection.  We also
sketch how the {\em breathing\/} construction of Knutson, Tao, and
Woodward \cite{knutson.tao.ea:honeycomb} can be carried out using
mutations; this application was pointed out by the referee.

This paper is organized as follows.  In Section~\ref{sec:equiv} we
state our puzzle formula for the equivariant cohomology of two-step
flag varieties and specialize it to the equivariant Gromov-Witten
invariants of Grassmannians.  Section~\ref{sec:recurse} explains the
recursive identity required to prove the formula.
Section~\ref{sec:mutations} starts with an informal introduction of
the mutation algorithm, after which we give the precise definitions
and prove that the mutation algorithm works as required.  Finally,
Section~\ref{sec:aura} defines auras associated to various objects in
puzzles and uses this concept to prove our main result.

Parts of the work on writing this paper was carried out during a visit
to the University of Copenhagen during the summer of 2013.  We thank
the Mathematics Department in Copenhagen for their hospitality and for
providing a friendly and stimulating environment.  We thank Andrew
Kresch and Kevin Purbhoo for many inspiring discussions about puzzles.
We finally thank the referee for many helpful comments and
suggestions, including the above-mentioned relation to breathing of
puzzles.

%%% Local Variables: 
%%% mode: latex
%%% TeX-master: "puzzle2eq.tex"
%%% End: 

%% file: equiv.tex
\section{The equivariant puzzle formula}\label{sec:equiv}

\subsection{Two-step flag varieties}\label{ssec:2step}

Fix integers $0 \leq a \leq b \leq n$ and let $X = \Fl(a,b;n)$ be the
variety of two-step flags $A \subset B \subset \C^n$ such that
$\dim(A)=a$ and $\dim(B)=b$.  A {\em 012-string\/} for $X$ is a
sequence $u = (u_1,u_2,\dots,u_n) \in \Z^n$ consisting of $a$ zeros,
$b-a$ ones, and $n-b$ twos.  The Schubert varieties in $X$ are indexed
by these 012-strings.  Let $e_1,e_2,\dots,e_n$ be the standard basis
for $\C^n$, let $\Bo \subset \GL_n(\C)$ be the Borel subgroup of upper
triangular matrices, and let $\Bo^- \subset \GL_n(\C)$ be the opposite
Borel subgroup of lower triangular matrices.  We also let $T = \Bo \cap
\Bo^-$ be the maximal torus of diagonal matrices.  Given any 012-string
$u$ for $X$, we define the subspaces $A_u = \Span_\C\{ e_i \mid u_i=0
\}$ and $B_u = \Span_\C\{ e_i \mid u_i \leq 1 \}$ of $\C^n$.  Then
$(A_u,B_u)$ is a point in $X$, and the $T$-fixed points in $X$ are
exactly the points of this form.  Let $X_u = \overline{\Bo.(A_u,B_u)}$
the {\em Schubert variety\/} defined by $u$, and let $X^u =
\overline{\Bo^-.(A_u,B_u)}$ the {\em opposite Schubert variety\/}
defined by $u$.  We have $\dim(X_u) = \codim(X^u,X) = \ell(u) =
\#\{i<j \mid u_i>u_j\}$.

Let $H^*_T(X;\Z)$ denote the $T$-equivariant cohomology ring of $X$.
An introduction to this ring can be found in
\cite{anderson.fulton:equivariant}.  Each $T$-stable closed subvariety
$Z \subset X$ defines an equivariant class $[Z] \in H^{2d}_T(X;\Z)$,
where $d = \codim(Z,X)$.  Pullback along the structure morphism $X \to
\{\pt\}$ gives $H^*_T(X;\Z)$ the structure of an algebra over the ring
$\La := H^*_T(\pt;\Z)$, and $H^*_T(X;\Z)$ is a free $\La$-module with
a basis consisting of the Schubert classes $[X^u]$ indexed by all
012-strings for $X$.  The {\em equivariant Schubert structure
  constants\/} of $X$ are the unique classes $C^w_{u,v} \in \La$
defined by the equation
\begin{equation}\label{eqn:const}
[X^u] \cdot [X^v] = \sum_w C^w_{u,v} [X^w] \ \in H^*_T(X;\Z) \,,
\end{equation}
where the sum is over all 012-strings $w$ for $X$.  Let $\int_X :
H^*_T(X;\Z) \to \La$ denote the pushforward along the map $X \to
\{\pt\}$.  For arbitrary 012-strings $u$ and $v$ for $X$ we then have
$\int_X [X^u] \cdot [X_v] = \delta_{u,v}$.  It follows that the
equivariant structure constants of $X$ are given by
\[
C^w_{u,v} \ = \ \int_X [X^u] \cdot [X^v] \cdot [X_w] \,.
\]

Each basis element $e_i$ for $\C^n$ defines a one-dimensional
$T$-representation $\C e_i$, where the action is given by
$(t_1,\dots,t_n).e_i = t_i e_i$.  This representation can be regarded
as a $T$-equivariant line bundle over a point.  We let $y_i = -c_1(\C
e_i) \in \La$ be the corresponding equivariant Chern class with
negated sign.\footnote{The sign ensures consistency with standard
  notation for double Schubert polynomials.}  We then have $\La =
\Z[y_1,\dots,y_n]$.  Since $H^*_T(X;\Z)$ is a graded ring, it follows
that each structure constant $C^w_{u,v} \in \La$ is a homogeneous
polynomial of total degree $\ell(u)+\ell(v)-\ell(w)$.  The constants
$C^w_{u,v} \in \Z$ for which $\ell(u)+\ell(v)=\ell(w)$ are the
structure constants of the ordinary cohomology ring $H^*(X;\Z)$.
These constants are given by the cohomological puzzle rule proved in
\cite{buch.kresch.ea:puzzle}.  A result of Graham
\cite{graham:positivity} asserts that every equivariant structure
constant $C^w_{u,v} \in \La$ is a polynomial with positive
coefficients in the differences $y_{i+1}-y_i$, i.e.\ we have
$C^w_{u,v} \in \Z_{\geq 0}[y_2-y_1,\dots,y_n-y_{n-1}]$.  We proceed to
state our manifestly positive formula for these constants.

\subsection{Equivariant puzzles}\label{ssec:equiv_puzzles}

A {\em puzzle piece\/} is a figure from the following list.
\[
\begin{split}
&
\pic{.7}{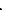} \ \ 
\pic{.7}{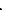} \ \ 
\pic{.7}{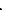} \ \ 
\pic{.7}{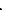} \ \ 
\pic{.7}{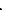} \ \ 
\pic{.7}{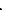} \ \ 
\pic{.7}{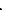} \ \ 
\pic{.7}{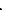} \\
&
\pic{.7}{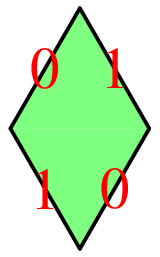} \ \ 
\pic{.7}{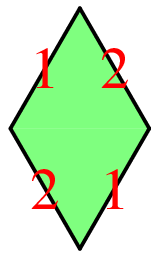} \ \ 
\pic{.7}{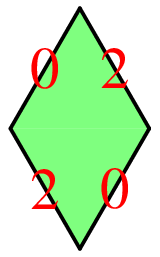} \ \ 
\pic{.7}{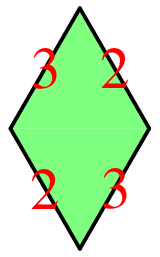} \ \ 
\pic{.7}{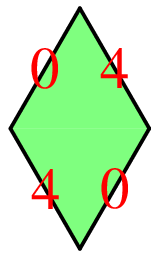} \ \ 
\pic{.7}{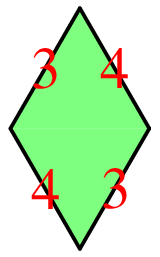} \ \ 
\pic{.7}{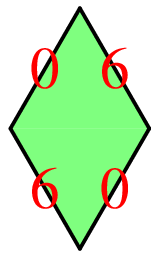} \ \ 
\pic{.7}{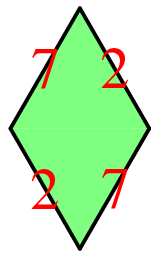}
\end{split}
\]
The triangular puzzle pieces come from the cohomological puzzle rule
\cite{buch.kresch.ea:puzzle}, which was originally conjectured by
Knutson \cite{knutson:conjectural}.  In Knutson's notation the side
labels were parenthesized strings of the integers 0, 1, and 2.  The
labels that are greater than two can be translated to such strings as
follows:
\[
3=10, \ 4=21, \ 5=20, \ 6=2(10), \ 7=(21)0 \,.
\]
The labels 0, 1, 2 are called {\em simple\/} and the other labels 3,
4, 5, 6, 7 are called {\em composed}.  Notice that a triangular puzzle
piece is uniquely determined if the labels on two of its sides are
known.

The rhombus-shaped puzzle pieces are called {\em equivariant puzzle
  pieces}.  The first equivariant piece comes from Knutson and Tao's
puzzle formula for the equivariant structure constants of
Grassmannians \cite{knutson.tao:puzzles}.  In fact, the first five
equivariant pieces are very natural from the statement of this formula
together with the cohomological puzzle rule for two-step flag
varieties.  The last three equivariant pieces are more surprising, as
each of them embeds the same simple label on all sides, which appears
to violate the philosophy of Knutson's original conjecture
\cite{knutson:conjectural}.  Puzzle pieces may be rotated arbitrarily,
but they may not be reflected.  An equivariant puzzle piece is called
{\em vertical\/} if it is oriented as in the above list.

Define a {\em triangular puzzle\/} to be any equilateral triangle made
from puzzle pieces with matching labels, i.e.\ any two puzzle pieces
next to each other assign the same label to the side that they share.
We also demand that all labels on the boundary of the triangle are
simple, and that the triangle is `right side up', i.e.\ its bottom
border is a horizontal line segment.  The sides of the puzzle pieces
in a puzzle are called {\em puzzle edges}, and the three sides of the
boundary of the puzzle are called {\em border segments}.  We will say
that a triangular puzzle $P$ has {\em boundary\/} $\border^{u,v}_w$,
also written as $\partial P = \border^{u,v}_w$, if $u$ is the string
of labels on the left border segment, $v$ is the string of labels on
the right border segment, and $w$ is the string of labels on the
bottom border segment, all read in left to right order.

The triangular puzzle $P$ is called an {\em equivariant puzzle\/} for
$X$ if all its equivariant pieces are vertical and the boundary of $P$
is $\border^{u,v}_w$ where $u$, $v$, and $w$ are 012-strings for $X$.
The composed labels in any puzzle are uniquely determined by the
simple labels, so we will often omit them in pictures of puzzles.  The
following are two pictures of the same equivariant puzzle for the
variety $\Fl(2,4;6)$, with and without the composed labels.  This
puzzle has boundary $\border^{u,v}_w$ where $u = (1,1,0,2,0,2)$, $v =
(0,2,1,2,1,0)$, and $w = (1,2,0,2,1,0)$.
\[
\pic{.5}{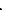} \hspace{10mm} \pic{.5}{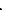}
\]

Given an equivariant puzzle $P$ for $X$, we number the edges of the
bottom border segment from 1 to $n$, starting from the left.  Each
equivariant puzzle piece $q$ in $P$ has a {\em weight\/} defined by
$\weight(q) = y_j - y_i$, where $i$ is the number of the bottom edge
obtained by following a south-west line from $q$, and $j$ is the
bottom edge number obtained by following a south-east line from $q$.
For example, the following equivariant puzzle piece has weight
$y_6-y_3$.
\[
\psfrag{1}{1}
\psfrag{2}{2}
\psfrag{3}{3}
\psfrag{4}{4}
\psfrag{5}{5}
\psfrag{6}{6}
\psfrag{7}{$n$}
\pic{.4}{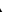}
\]
The weight of the equivariant puzzle $P$ is the product of the weights
of all equivariant puzzle pieces in $P$.

Our main result is the following Graham-positive combinatorial formula
for the equivariant Schubert structure constants of $X$, which
generalizes both Knutson and Tao's equivariant rule for Grassmannians
\cite{knutson.tao:puzzles} and the cohomological puzzle rule for
two-step flag varieties \cite{buch.kresch.ea:puzzle}.  We conjectured
this formula more than 10 years ago, and our conjecture was printed in
Co\c{s}kun and Vakil's survey \cite{coskun.vakil:geometric}.\footnote{The
  statement in \cite{coskun.vakil:geometric} lacks two of the
  equivariant puzzle pieces; the author of the present paper is
  entirely responsible for this.}

\begin{thm}\label{thm:puzzle}
  Let $u$, $v$, and $w$ be 012-strings for the two-step flag variety
  $X = \Fl(a,b;n)$.  Then the equivariant Schubert structure constant
  $C^w_{u,v} \in \La$ is given by
  \[
  C^w_{u,v} \ = \ \sum_{\partial P = \triangle^{u,v}_w} \weight(P)
  \]
  where the sum is over all equivariant puzzles $P$ for $X$ with
  boundary $\border^{u,v}_w$.
\end{thm}

\begin{example}
  For $X = \Fl(2,4;5)$ we have $[X^{01201}] \cdot [X^{10102}] =
  [X^{12010}] + [X^{11200}] + (y_4-y_1)[X^{12001}] +
  (y_5+y_4-y_3-y_1)[X^{10210}] + (y_4-y_3)(y_4-y_1)[X^{10201}]$.  The
  puzzles required to compute this product are:

  \pic{.5}{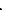} \ \ \pic{.5}{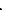} \ \ \pic{.5}{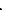} 

  \pic{.5}{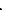} \ \ \pic{.5}{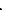} \ \ \pic{.5}{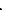}
\end{example}

\subsection{Equivariant quantum cohomology}\label{ssec:quantum}

Let $X = \Gr(m,n)$ be the Grassmann variety of $m$-dimensional
subspaces of $\C^n$.  By identifying $X$ with the variety
$\Fl(m,m;n)$, we may index the Schubert varieties in $X$ by strings
$\la=(\la_1,\dots,\la_n)$ containing $m$ zeros and $n-m$ twos.  Each
such string $\la$ can be identified with a {\em Young diagram\/}
contained in the rectangle with $m$ rows and $n-m$ columns.  More
precisely, $\la$ defines a path from the lower-left corner to the
upper-right corner of this rectangle, where the $i$-th step is
vertical if $\la_i=0$ and horizontal if $\la_i=2$, and $\la$ is
identified with the portion of the rectangle that is north-west of
this path.  For example, we identify $\la = (2,0,2,2,0,2,0,2)$ with
the Young diagram $\tableau{6}{{}&{}&{}&{}\\{}&{}&{}\\{}}$.
\[
\def\labv{\hspace{-4.5mm}0}
\def\labh{\raisebox{4.8mm}{\hspace{-9.0mm}2}}
\tableau{13}{
  [a]&{}&{}&{}&[LTr]\labv&[]\labh\\
  [a]&{}&{}&[LT]\labv&[r]\labh\\
  [ab]\raisebox{-6.5mm}{2}&[LTb]\labv&[Tb]\labh&[b]\labh&[br]\\
}\vspace{2mm}
\]

Given a degree $d \in \N$, let $M_d = \Mb_{0,3}(X,d)$ denote the
Kontsevich moduli space of 3-pointed stable maps to $X$ of genus zero
and degree $d$.  This variety parametrizes morphisms of varieties $f :
C \to X$ defined on a tree of projective lines $C$ with three ordered
marked non-singular points, such that $f_*[C] = d\,
[X_{\tableau{4}{{}}}] \in H_2(X;\Z)$, and any component of $C$ that is
mapped to a single point in $X$ contains at least three special
points, where a special point is either marked or singular.  The
variety $M_d$ is equipped with {\em evaluation maps\/} $\ev_i : M_d
\to X$ for $1 \leq i \leq 3$, where $\ev_i$ sends a stable map to the
image of the $i$-th marked point in its domain.  We refer to
\cite{fulton.pandharipande:notes} for a careful construction of this
space.

The {\em equivariant quantum cohomology ring\/} $\QH_T(X)$ is an
algebra over the ring $\Lambda[q]$, which as a module is defined by
$\QH_T(X) = H_T^*(X;\Z) \otimes_\Lambda \Lambda[q]$.  The
multiplicative structure of $\QH_T(X)$ is determined by
\[
[X^\la] \star [X^\mu] = \sum_{\nu, d \geq 0} N^{\nu,d}_{\la,\mu}\, q^d
\, [X^\nu] \,,
\]
where the structure constants $N^{\nu,d}_{\la,\mu} \in \La$ are the
{\em equivariant Gromov-Witten invariants\/} defined by
\[
N^{\nu,d}_{\la,\mu} \ = \ \int_{M_d} \ev_1^*[X^\la] \cdot
\ev_2^*[X^\mu] \cdot \ev_3^*[X_\nu] \,.
\]
It is a non-trivial fact that this construction defines an associative
ring \cite{ruan.tian:mathematical, kontsevich.manin:gromov-witten,
  kim:equivariant}.  The structure constant $N^{\nu,d}_{\la,\mu}$ is a
homogeneous polynomial in $\La = \Z[y_1,\dots,y_n]$ of total degree
$|\la|+|\mu|-|\nu|-nd$.  If this degree is zero, then
$N^{\nu,d}_{\la,\mu}$ is the number of stable maps $f \in M_d$ for
which $\ev_1(f)$, $\ev_2(f)$, and $\ev_3(f)$ belong to (fixed) general
translates of the Schubert varieties $X^\la$, $X^\mu$, and $X_\nu$.

In \cite{buch:quantum} we introduced the {\em kernel\/} and {\em
  span\/} of a rational curve in a Grassmann variety as a tool to
study its Gromov-Witten invariants.  The kernel of a stable map $f : C
\to X$ is defined as the intersection of the $m$-planes in its image,
and the span of $f$ is the linear span of these $m$-planes:
\[
\Ker(f) = \bigcap_{V \in f(C)} V 
\text{ \ \ \ \ \ and \ \ \ \ \ }
\Span(f) = \sum_{V \in f(C)} V \,.
\]
Define the two-step flag variety $Y_d = \Fl(m-d,m+d;n)$ and the
three-step flag variety $Z_d = \Fl(m-d,m,m+d;n)$; these varieties can
be regarded as empty if $d > \min(m,n-m)$.  Let $p : Z_d \to X$ and $h
: Z_d \to Y_d$ be the projections.  It was proved in
\cite{buch.kresch.ea:gromov-witten} that, when $N^{\nu,d}_{\la,\mu}$
has degree zero, the map $f \mapsto (\Ker(f),\Span(f))$ defines an
explicit bijection between the set of stable maps counted by
$N^{\nu,d}_{\la,\mu}$ and the set of points in the intersection of
general translates of the Schubert varieties $h(p^{-1}(X^\la))$,
$h(p^{-1}(X^\mu))$, and $h(p^{-1}(X_\nu))$ in $Y_d$.  It follows that
$N^{\nu,d}_{\la,\mu}$ is equal to a classical triple intersection
number of Schubert varieties in $Y_d$.  The following equivariant
generalization of this result was obtained in
\cite[Thm.~4.2]{buch.mihalcea:quantum}.

\begin{thm}[\cite{buch.kresch.ea:gromov-witten,
    buch.mihalcea:quantum}]\label{thm:quantumclassical}
  We have $N^{\nu,d}_{\la,\mu} = \int_{Y_d} h_*p^*[X^\la] \cdot
  h_*p^*[X^\mu] \cdot h_*p^*[X_\nu]$.
\end{thm}

Let $J^d(\la)$ denote the 012-string obtained from $\la$ by replacing
the first $d$ occurrences of $2$ and the last $d$ occurrences of $0$
with $1$.  For example, we obtain $J^2((2,0,2,2,0,2,0,2)) =
(1,0,1,2,1,2,1,2)$.  We then have $h(p^{-1}(X^\la)) = Y_d^{J^d(\la)}$,
i.e.\ $h(p^{-1}(X^\la))$ is the (opposite) Schubert variety in $Y_d$
defined by the 012-string $J^d(\la)$.  Furthermore, the varieties
$p^{-1}(X^\la)$ and $h(p^{-1}(X^\la))$ have the same dimension if and
only if the first $d$ occurrences of $2$ in $\la$ come before the last
$d$ occurrences of $0$; equivalently, the Young diagram of $\la$
contains a $d \times d$ rectangle.  Let $\la^\vee =
(\la_n,\la_{n-1},\dots,\la_1)$ denote the string $\la$ in reverse
order.  Then $X_\la$ is a translate of $X^{\la^\vee}$.  We obtain the
following consequence of Theorem~\ref{thm:puzzle} and
Theorem~\ref{thm:quantumclassical}.

\begin{cor}
  The Gromov-Witten invariant $N^{\nu,d}_{\la,\mu}$ is non-zero only
  if each of the Young diagrams of $\la$, $\mu$, and $\nu^\vee$
  contains a $d \times d$ rectangle.  In this case we have
  \[
  N^{\nu,d}_{\la,\mu} \ = \ \sum_P \weight(P)
  \]
  where the sum is over all equivariant puzzles $P$ for $Y_d$ with
  boundary $\border^{J^d(\la),J^d(\mu)}_{J^d(\nu^\vee)^\vee}$.
\end{cor}
\begin{proof}
  If the Young diagram of $\la$, $\mu$, or $\nu^\vee$ does not contain
  a $d \times d$ rectangle, then one of the classes $h_*p^*[X^\la]$,
  $h_*p^*[X^\mu]$, or $h_*p^*[X_\nu]$ is equal to zero.  On the other
  hand, if each of these Young diagrams contain a $d \times d$
  rectangle, then we obtain
  \[
  \begin{split}
  N^{\nu,d}_{\la,\mu} \ &= \ 
  \int_{Y_d} h_*p^*[X^\la] \cdot h_*p^*[X^\mu] \cdot h_*p^*[X_\nu] \\
  &= \ 
  \int_{Y_d} [Y_d^{J^d(\la)}] \cdot [Y_d^{J^d(\mu)}] \cdot 
  [(Y_d)_{J^d(\nu^\vee)^\vee}] \ = \ 
  C^{J^d(\nu^\vee)^\vee}_{J^d(\la),J^d(\mu)} \,,
  \end{split}
  \]
  and the result follows from Theorem~\ref{thm:puzzle}.
\end{proof}

\begin{example}
  In the equivariant quantum cohomology ring of $X = \Gr(2,5)$ we have
  \[
  \begin{split}
    & [X^{\tableau{4}{{}&{}\\{}}}\,] \star
    [X^{\tableau{4}{{}&{}&{}\\{}}}\,] \ \ = \\
    & \ \ (y_5-y_3)(y_5-y_1)(y_2-y_1)\,
    [X^{\tableau{4}{{}&{}&{}\\{}}}\,] \ + 
    \ (y_5-y_1)^2\, [X^{\tableau{4}{{}&{}&{}\\{}&{}}}\,] \ + \
    (y_5-y_1)\,[X^{\tableau{4}{{}&{}&{}\\{}&{}&{}}}\,] \ + \\
    & \ \ (y_5-y_3)(y_2-y_1)\, q \ + \
    (y_5-y_1)\, q\, [X^{\tableau{4}{{}}}\,] \ + \ q\,
    [X^{\tableau{4}{{}\\{}}}\,] \ + \ q\, [X^{\tableau{4}{{}&{}}}\,]
    \,.
  \end{split}
  \]
  The last four terms involving $q$ are accounted for by the
  following puzzles.
  
  \pic{.5}{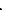} \ \ 
  \pic{.5}{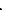} \ \ 
  \pic{.5}{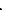}

  \pic{.5}{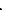} \ \ 
  \pic{.5}{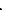} \ \ 
  \pic{.5}{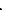}

  \pic{.5}{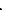} \ \ 
  \pic{.5}{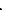}
\end{example}

%%% Local Variables: 
%%% mode: latex
%%% TeX-master: "puzzle2eq.tex"
%%% End: 

%% file: recurse.tex
\section{Recursive equations}\label{sec:recurse}

An observation that originates in Molev and Sagan's work
\cite{molev.sagan:littlewood-richardson} shows that all the
equivariant structure constants $C^w_{u,v}$ are determined by the
structure constants of the form $C^w_{w,w}$ by a set of recursive
identities.  This observation was used to prove the equivariant puzzle
rule for Grassmannians \cite{knutson.tao:puzzles}, and it was extended
to equivariant quantum cohomology in \cite{mihalcea:equivariant*1}.
Molev and Sagan's recursions apply to the equivariant structure
constants of any homogeneous space $Y$, and in general involves $2r$
families of identities where $r$ is the rank of the Picard group of
$Y$.  The structure constants of the form $C^w_{w,w}$ are given by a
formula of Kostant and Kumar \cite{kostant.kumar:nil*1}.  In this
section we will arrange the recursive identities into a single family,
focusing on the two-step flag variety $X = \Fl(a,b;n)$.

For this purpose we will work with $T$-equivariant cohomology with
coefficients in the polynomial ring $R = \C[\delta_0, \delta_1,
\delta_2]$.  The variables of this ring correspond to the simple
puzzle labels, and the field of complex numbers $\C$ will be utilized
as a two-dimensional real plane where puzzle angles can be encoded.
This will later make it possible to use the triangular geometry of
puzzles to prove the required recursive identities.  We have
$H^*_T(\pt;R) = R[y] := R[y_1,\dots,y_n]$ where $y_i = -c_1(\C e_i)$,
and this ring contains $\La$ as a subring.  Furthermore, the ring
$H^*_T(X;R) = H^*_T(X;\Z) \otimes_\Z R$ is an $R[y]$-algebra with an
$R[y]$-basis consisting of the equivariant Schubert classes $[X^u$].
The defining equation (\ref{eqn:const}) for the equivariant structure
constants $C^w_{u,v} \in \La$ is also valid in $H^*_T(X;R)$.

The {\em Bruhat order\/} on the set of 012-strings for $X$ is defined
by $u \leq v$ if and only if $X_u \subset X_v$.  We will write $u \to
u'$ if $u'$ covers $u$ in the Bruhat order, i.e.\ we have $u \leq u'$
and $\ell(u') = \ell(u)+1$.  Equivalently, $u'$ can be obtained from
$u$ by replacing a connected subsequence in one of the following three
ways:
\[
(0,2^m,1) \to (1,2^m,0)
\text{ \ \ \ or \ \ \ }
(0,2) \to (2,0)
\text{ \ \ \ or \ \ \ }
(1,0^m,2) \to (2,0^m,1) \,.
\]
Here $x^m$ denotes a sequence of $m$ copies of $x$.  Given a covering
$u \to u'$ we set
\[
\delta\left(\frac{u}{u'}\right) = \delta_{u_i} - \delta_{u'_i} \,,
\]
where $i$ is the smaller index for which $u_i \neq u'_i$.  For
example, we have $\delta(\frac{10221}{11220}) = \delta_0-\delta_1$ and
$\delta(\frac{12021}{12201}) = \delta_0-\delta_2$.  Finally, given any
012-string $u$ for $X$ we define
\[ 
C_u = \sum_{i=1}^n \delta_{u_i} y_i \ \in R[y] \,.
\]
For example, $C_{01021} = \delta_0 y_1 + \delta_1 y_2 + \delta_0 y_3 +
\delta_2 y_4 + \delta_1 y_5$.

Set $\zeta = \exp(\pi i / 6) \in \C$.  Notice that the odd powers of
$\zeta$ are unit vectors perpendicular to puzzle edges.
\[
\psfrag{1}{$\zeta$}
\psfrag{3}{\raisebox{2pt}{$\zeta^3$}}
\psfrag{5}{\raisebox{1pt}{\!\!$\zeta^5$}}
\psfrag{7}{$\!\!\zeta^7$}
\psfrag{9}{$\zeta^9$}
\psfrag{0}{$\zeta^{11}$}
\pic{1}{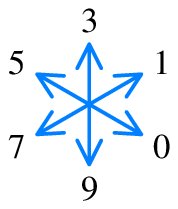}
\]

\begin{thm}\label{thm:recurse}
  The equivariant Schubert structure constants $C^w_{u,v}$ of the
  two-step partial flag variety $X = \Fl(a,b;n)$ satisfy the
  identities
  \begin{equation}\label{eqn:extreme}
    C^w_{w,w} = \prod_{i<j :\, w_i>w_j} (y_j-y_i)
  \end{equation}
  in $\La$ and
  \begin{multline}\label{eqn:recurse}
    (C_u \zeta^{11} + C_v \zeta^7 + C_w \zeta^3)\, C^w_{u,v} \ = \\
    \zeta^5 \sum_{u\to u'} \delta\left(\frac{u}{u'}\right) C^w_{u',v} +
    \zeta \sum_{v\to v'} \delta\left(\frac{v}{v'}\right) C^w_{u,v'} +
    \zeta^9 \sum_{w'\to w} \delta\left(\frac{w'}{w}\right) C^{w'}_{u,v}
  \end{multline}
  in $R[y]$, for all 012-strings $u$, $v$, and $w$ for $X$.
  Furthermore, the equivariant structure constants of $X$ are uniquely
  determined by these identities, i.e.\ any family of classes
  $C^w_{u,v}$ that satisfy (\ref{eqn:extreme}) and (\ref{eqn:recurse})
  are the equivariant structure constants of $X$.
\end{thm}

Theorem~\ref{thm:recurse} will be proved at the end of this section
after some additional notation has been introduced.  Given a class
$\Omega \in H^*_T(X;R)$ and a 012-string $u$ for $X$, we let $\Omega_u
\in R[y]$ denote the restriction of $\Omega$ to the $T$-fixed point
$(A_u,B_u)$.  The following identity is a special case of a formula of
Kostant and Kumar \cite[Prop.~4.24(a)]{kostant.kumar:nil*1}.

\begin{lemma}\label{lem:extreme}
  For any 012-string $w$ for $X$ we have $[X^w]_w = \ds\prod_{i<j :\,
    w_i>w_j}\! (y_j-y_i)$.
\end{lemma}
\begin{proof}
  It follows from \cite[Ex.~3.3.2]{fulton:intersection} that $[X^w]_w$
  is the top Chern class of the fiber of the normal bundle of $X^w$ in
  $X$ over the point $(A_w,B_w)$.  For $i \neq j$ we define maps
  $\gamma_{i,j} : \C \to \GL_n(\C)$ and $\varphi_{i,j} : \C \to X$ by
  $\gamma_{i,j}(s).e_k = e_k + s\, \delta_{jk}\, e_i$ for all $k$, and
  $\varphi_{i,j}(s) = \gamma_{i,j}(s).(A_w,B_w)$.  Here $\delta_{jk}$
  is Kronecker's delta.  The tangent space $T_w X$ of $X$ at
  $(A_w,B_w)$ has the basis $\{\varphi'_{i,j}(0) \mid w_i > w_j \}$,
  and $T_w X^w$ has basis $\{\varphi'_{i,j}(0) \mid w_i > w_j \text{
    and } i>j \}$.  It follows that the normal space $T_wX / T_w X^w$
  has basis $\{\varphi'_{i,j}(0) \mid w_i>w_j \text{ and } i < j \}$.
  Since the torus $T$ acts on $T_w X$ by $t.\varphi'_{i,j}(0) =
  t_i/t_j\, \varphi'_{i,j}(0)$, we obtain $c_1(\C\, \varphi'_{i,j}(0))
  = y_j-y_i$.  This proves the lemma.
\end{proof}

\begin{lemma}\label{lem:samecoef}
  Let $\Omega \in H^*_T(X;R)$, let $u$ be a 012-string for $X$, and
  consider the expansion $\Omega \cdot [X^u] = \sum_w d_w [X^w]$ where
  $d_w \in R[y]$.  Then $d_w$ is non-zero only if $u \leq w$, and we
  have $d_u = \Omega_u$.
\end{lemma}
\begin{proof}
  We have $d_w = \int_X \Omega \cdot [X^u] \cdot [X_w] = \int_X \Omega
  \cdot [X^u \cap X_w]$, and the intersection $X^u \cap X_w$ is
  non-empty if and only if $u \leq w$.  The last identity follows
  because $d_u = \int_X \Omega \cdot [(A_u,B_u)] = \Omega_u$.
\end{proof}

Lemma~\ref{lem:samecoef} implies that we have $C^w_{u,w} = [X^u]_w$
for arbitrary 012-strings $u$ and $w$ for $X$.  In particular, the
identity (\ref{eqn:extreme}) follows from Kostant and Kumar's formula
for $[X^w]_w$.  A formula for the more general restrictions $[X^u]_w$
has been proved by Andersen, Jantzen, and Soergel
\cite[App.~D]{andersen.jantzen.ea:representations} and by Billey
\cite{billey:kostant}.

Let $D_1 = X^{\la(1)}$ and $D_2 = X^{\la(2)}$ be the Schubert divisors
on $X$, defined by the 012-strings $\la(1) =
(0^{a-1},1,0,1^{b-a-1},2^{n-b})$ and $\la(2) =
(0^a,1^{b-a-1},2,1,2^{n-b-1})$.  We will work with the class $\cD =
(\delta_0-\delta_1)\,[D_1] + (\delta_1-\delta_2)\,[D_2] \in
H^*_T(X;R)$ that encodes both of these divisors.  We also set $C_0 =
C_{(0^a,1^{b-a},2^{n-b})} \in R[y]$.

\begin{lemma}\label{lem:chevalley}
  For any 012-string $u$ for $X$ we have
  \[
  \cD \cdot [X^u] = (C_u-C_0)\,[X^u] + \sum_{u\to u'}
  \delta\left(\frac{u}{u'}\right) [X^{u'}] \,.
  \]
\end{lemma}
\begin{proof}
  The equivariant ring $H^*_T(X;R)$ has a natural grading by complex
  codimension given by $\deg\, [X^w] = \ell(w)$, $\deg(y_i)=1$, and
  $\deg(\delta_j)=0$.  It therefore follows from
  Lemma~\ref{lem:samecoef} that, if the coefficient of $[X^w]$ is
  non-zero in the expansion of $\cD \cdot [X^u]$, then we have either
  $u=w$ or $u\to w$.  Recall that the classical Monk/Chevalley formula
  \cite{chevalley:decompositions, monk:geometry} states that the
  product $[D_1] \cdot [X^u]$ in the ordinary cohomology ring
  $H^*(X;\Z)$ is equal to the sum of all classes $[X^w]$ for which $u
  \to w$ and $\delta(\frac{u}{w}) \in \{\delta_0-\delta_1,
  \delta_0-\delta_2\}$, and the product $[D_2] \cdot [X^u]$ is the sum
  of all classes $[X^w]$ for which $u \to w$ and $\delta(\frac{u}{w})
  \in \{\delta_0-\delta_2, \delta_1-\delta_2\}$.  It follows that the
  coefficient of $[X^w]$ in $\cD \cdot [X^u]$ is equal to
  $\delta(\frac{u}{w})$ whenever $u \to w$.  It remains to show that
  $\cD_u = C_u-C_0$.  This has been proved in higher generality by
  Kostant and Kumar \cite[Prop.~4.24(c)]{kostant.kumar:nil*1}.  For
  completeness we give the following argument from \cite[\S
  8]{buch.mihalcea:curve}.

  For $m\leq n$ we set $\C^m = \Span_\C\{e_1,\dots,e_m\} \subset \C^n$
  and $\C^m_X = \C^m \times X$.  There is a natural sequence of vector
  bundles over $X$ given by $\cA \subset \cB \subset \C^n_X \to \C^b_X
  \to \C^a_X$, where $\cA$ and $\cB$ are the tautological subbundles
  on $X$ and the last two maps are the projections to the first $b$
  and $a$ coordinates in $\C^n$.  The Schubert divisors on $X$ are the
  zero sections $D_1 = Z(\bigwedge^a \cA \to \bigwedge^a \C^a_X)$ and
  $D_2 = Z(\bigwedge^b \cB \to \bigwedge^b \C^b_X)$.  It follows that
  \[
  \begin{split}
    \cD_u 
    &= (\delta_0-\delta_1)\,\left(c_1(\C^a_X) - c_1(\cA)\right)_u
       + (\delta_1-\delta_2)\,\left(c_1(\C^b_X) - c_1(\cB)\right)_u \\
    &= \delta_0 \big( c_1(\C^a) - c_1(A_u) \big) + 
       \delta_1 \big( c_1(\C^b/\C^a) - c_1(B_u/A_u) \big) \\
    & \ \ \ \ \ + \delta_2 \big( c_1(\C^n/\C^b) - c_1(\C^n/B_u) \big) \\
    &= C_u - C_0 \,.
  \end{split}
  \]
  This completes the proof.
\end{proof}

\begin{proof}[Proof of Theorem~\ref{thm:recurse}]
  The identity (\ref{eqn:extreme}) follows from
  Lemma~\ref{lem:extreme} and Lemma~\ref{lem:samecoef}.  To prove
  (\ref{eqn:recurse}), we use Lemma~\ref{lem:chevalley} and the
  equivariant structure constants of $X$ to expand both sides of the
  associativity relation $(\cD \cdot [X^u]) \cdot [X^v] = \cD \cdot
  ([X^u] \cdot [X^v])$ in the basis of Schubert classes.  Since the
  coefficient of $[X^w]$ in both sides is the same, we obtain the
  identity
  \begin{equation}\label{eqn:rec1}
    (C_u-C_w)\, C^w_{u,v} = 
    \sum_{w' \to w} \delta\left(\frac{w'}{w}\right) C^{w'}_{u,v}
    - \sum_{u\to u'} \delta\left(\frac{u}{u'}\right) C^w_{u',v} \,.
  \end{equation}
  Similarly, the relation $[X^u] \cdot ([X^v] \cdot \cD) =
  ([X^u]\cdot [X^v]) \cdot \cD$ implies the identity
  \begin{equation}\label{eqn:rec2}
    (C_v-C_w)\, C^w_{u,v} =
    \sum_{w'\to w} \delta\left(\frac{w'}{w}\right) C^{w'}_{u,v}
    - \sum_{v\to v'} \delta\left(\frac{v}{v'}\right) C^w_{u,v'} \,.
  \end{equation}
  Finally, the identity (\ref{eqn:recurse}) is obtained by multiplying
  both sides of (\ref{eqn:rec1}) with $\zeta^{11}$, multiplying both
  sides of (\ref{eqn:rec2}) with $\zeta^7$, and adding the resulting
  equations.

  We next observe that $C_u-C_w$ is non-zero whenever $u \neq w$, and
  $C_v-C_w$ is non-zero whenever $v \neq w$.  Since both $C_u-C_w$ and
  $C_v-C_w$ are elements of the polynomial ring
  $\Z[\delta_0,\delta_1,\delta_2,y_1,\dots,y_n]$, and the powers
  $\zeta^{11}$ and $\zeta^7$ are linearly independent over this ring,
  it follows that the factor $(C_u \zeta^{11} + C_v \zeta^7 + C_w
  \zeta^3) = (C_u-C_w) \zeta^{11} + (C_v-C_w) \zeta^7$ of
  (\ref{eqn:recurse}) is non-zero whenever $u\neq w$ or $v\neq w$.

  We finally prove that the equivariant Schubert structure constants
  of $X$ are uniquely determined by (\ref{eqn:extreme}) and
  (\ref{eqn:recurse}) by descending induction on $\deg(C^w_{u,v}) =
  \ell(u)+\ell(v)-\ell(w)$.  The basis step is vacuous because
  $\deg(C^w_{u,v}) \leq 2 \dim(X)$.  For the inductive step, let $u$,
  $v$, and $w$ be given.  If $u=v=w$, then the constant $C^w_{u,v}$ is
  determined by (\ref{eqn:extreme}).  Otherwise notice that all
  structure constants appearing on the right side of equation
  (\ref{eqn:recurse}) have degree equal to $\deg(C^w_{u,v})+1$, so
  these constants are uniquely determined by the induction hypothesis.
  Since $R[y]$ is a domain and $(C_u \zeta^{11} + C_v \zeta^7 + C_w
  \zeta^3) \neq 0$, we deduce that $C^w_{u,v}$ is uniquely determined
  as well.
\end{proof}

\begin{remark}
  The real scalar product of two vectors $x,y \in \C$ is defined by
  $(x,y) = \Re(x \overline{y})$, and this scalar product has an
  $\R$-linear extension to the ring $R[y]$.  By taking the scalar
  product of both sides of equation (\ref{eqn:recurse}) with the
  vector $\zeta^{10}$, we recover the identity (\ref{eqn:rec1})
  associated to the relation $(\cD \cdot[X^u])\cdot[X^v] =
  \cD\cdot([X^u]\cdot[X^v])$, and by taking the scalar product with
  $\zeta^8$, we recover the identity (\ref{eqn:rec2}) associated to
  the relation $[X^u]\cdot([X^v]\cdot\cD) =
  ([X^u]\cdot[X^v])\cdot\cD$.  One may check that the scalar product
  of equation (\ref{eqn:recurse}) with $1 \in \C$ results in an
  identity associated to the relation $([X^u]\cdot\cD)\cdot[X^v] =
  [X^u]\cdot(\cD\cdot[X^v])$.
\end{remark}

%%% Local Variables: 
%%% mode: latex
%%% TeX-master: "puzzle2eq.tex"
%%% End: 

%% file: mutate.tex
\section{Mutations of puzzles}\label{sec:mutations}

\subsection{Puzzles}\label{ssec:puzzles}

Define a {\em puzzle\/} to be any hexagon made from puzzle pieces with
matching side labels, such that all boundary labels are simple.  In
contrast to the conventions used in \cite{knutson.tao:puzzles} we
allow all puzzle pieces to be rotated arbitrarily, including
equivariant pieces.  This means that rotations of puzzles are again
puzzles, which will be exploited to simplify constructions and proofs.
We shall work only with puzzles whose edges are parallel to the sides
of a right-side-up triangle.  Define the {\em dual\/} of a puzzle to
be the result of reflecting it in a vertical line and applying the
following substitution to its labels:
\[
0 \mapsto 2 \ , \ 
1 \mapsto 1 \ , \ 
2 \mapsto 0 \ , \ 
3 \mapsto 4 \ , \ 
4 \mapsto 3 \ , \ 
5 \mapsto 5 \ , \ 
6 \mapsto 7 \ , \ 
7 \mapsto 6 \,.
\]
For example, the following two puzzles are dual to each other.
\[
\pic{.7}{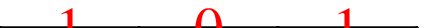} \ \ \ \ \pic{.7}{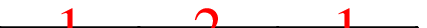}
\]

The line segments that make up the boundary of a puzzle are called
{\em border segments}.  We allow border segments to have length zero.
In particular, the shape of a puzzle may be an equilateral triangle.

A {\em gashed puzzle\/} is a hexagon made of puzzle pieces, not
necessarily with matching side labels, but still with simple boundary
labels.  The puzzle edges where the labels do not match are called
{\em gashes}.  We think about gashes as edges that have two labels,
one on each side.  We also allow gashes on the boundary of a gashed
puzzle, by artificially imposing an extra label on the far side of a
boundary edge.  The following gashed puzzle has two gashes.
\[
\pic{.7}{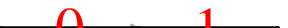}
\]
We will use the textual notation {\psfrag{0}{$b$} \psfrag{1}{$a$}
  $\ds\frac{\,a\,}{\,b\,} = \raisebox{-1.5mm}{\pic{.6}{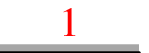}}$, $b/a
  = \raisebox{-3.5mm}{\pic{.6}{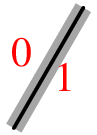}}$, and $a\backslash b =
  \raisebox{-3.5mm}{\pic{.6}{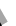}}$} for gashes of the three
possible orientations.  For example, the two gashes in the above
example are denoted $5/0$ and $0\backslash 2$.

\subsection{Introduction to mutations}\label{ssec:mut_intro}

The main new combinatorial construction in this paper is an algorithm
called {\em
  mutation of puzzles}.  Before we state the precise definition, we
will give a more informal introduction by working through some
examples.  Consider the following gashed puzzle, where both gashes are
located on the south-west border segment.
\[
\pic{.7}{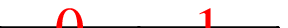}
\]
Such a pair of gashes can be introduced on an ungashed puzzle if one
wishes to change the labels of a border segment.  For example, the
bottom gash $0\backslash 2$ indicates that the label 2 should be
changed to 0.  We will always make such a change by replacing the
puzzle piece that contributes the unwanted label of the gash with a
new piece of the same shape, and this new piece must be chosen such
that only one new gash is created by the replacement.  It is a
fundamental observation that there is always {\em at most one\/}
puzzle piece that satisfies this requirement.  In our example we must
replace the puzzle piece \raisebox{-4mm}{\pic{.6}{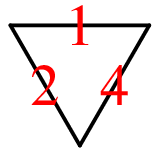}} with either
\raisebox{-4mm}{\pic{.6}{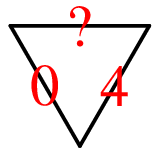}} or
\raisebox{-4mm}{\pic{.6}{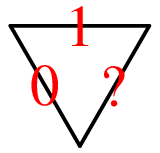}}, where the question marks can be
arbitrary labels, and the only possible choice is
\raisebox{-4mm}{\pic{.6}{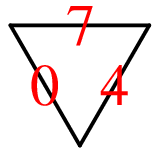}}.  The replacement makes the bottom
gash move to the top side of the replaced puzzle piece.  We say that
the gash has been {\em propagated}.  After the gash has been moved,
the process can be repeated to propagate it one more step.  However,
after two propagations have been carried out, no further propagations
are possible.  The steps are displayed in the following sequence of
gashed puzzles, where we have also indicated the direction in which
each gash is supposed to move.
\[
\pic{.7}{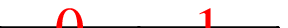}
\ \ \ \ 
\pic{.7}{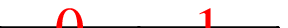}
\ \ \ \ 
\pic{.7}{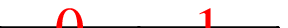}
\]
Propagation of the second gash gives the following continuation.
\[
\pic{.7}{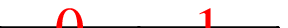}
\ \ \ \ 
\pic{.7}{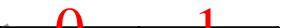}
\ \ \ \ 
\pic{.7}{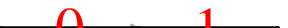}
\]
At this point both gashes are stuck at two sides of the same puzzle
piece.  We will show in Theorem~\ref{thm:mutate} below that this is no
coincidence.  At this time we change the labels of the gashed edges to
what the gashes suggest.  The result is the following {\em flawed
  puzzle}, where one of the small triangles is a {\em temporary puzzle
  piece}.  A temporary puzzle piece is analogous to an empty box in a
Young tableau during a sequence of jeu de taquin slides; in fact, it
is also possible to regard the temporary piece as a hole in the puzzle
where no valid puzzle piece will fit.
\[
\pic{.7}{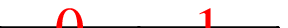}
\]

We would like to end up with a valid puzzle made from the puzzle
pieces listed in Section~\ref{sec:equiv}, which means that we have to
get rid of the temporary puzzle piece.  A temporary puzzle piece can
be {\em resolved\/} in three different ways, each of which preserves
one of its sides and replaces the other two sides with gashes.  This
is done by replacing the temporary piece with a valid piece that has
the same label on the side that is preserved.  When we work with
two-step puzzles, this valid piece is the unique puzzle piece whose
largest label is the preserved label.

In our example, if we choose to preserve the bottom side with label 7
of the temporary piece, then we replace this piece with
\raisebox{-4.5mm}{\pic{.6}{u740}}.  The resulting gashes can be
propagated as follows.
\[
\pic{.7}{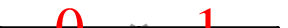}
\ \ \ \ 
\pic{.7}{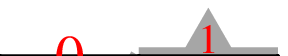}
\ \ \ \ 
\pic{.7}{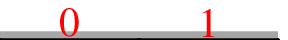}
\]

On the other hand, if we preserve the right side with label 3 of the
temporary puzzle piece, then the temporary piece is replaced with
\raisebox{-4.5mm}{\pic{.6}{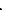}}, and by propagating the resulting
gashes we recover the puzzle that we started with.

Finally, if we choose to preserve the left side with label 5 of the
temporary piece, then this piece is replaced with
\raisebox{-4.5mm}{\pic{.6}{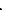}}, and the resulting gashes can be
propagated as follows (skipping some steps).
\[
\pic{.7}{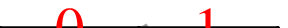}
\ \ \ \ 
\pic{.7}{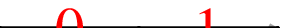}
\ \ \ \ 
\pic{.7}{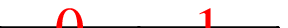}
\]
The middle picture shows the gashes at the positions where they get
stuck, which is on two sides of an equivariant puzzle piece.  We can
change the labels of these edges to what the gashes suggest by
replacing the equivariant piece with a rhombus made from two
triangular puzzle pieces.  Following \cite{knutson.tao:puzzles}, we
call this rhombus a {\em scab}, and it has been colored light blue to
{\em mark\/} its position.  This allows the mutation to be inverted.

\subsection{Propagation of gashes}\label{ssec:propagate}

We now give a detailed definition of the mutation algorithm, starting
with several related concepts.
Define a {\em directed gash\/} to be a gash together with a direction
perpendicular to its edge.  In pictures we will indicate the direction
with a gray arrow.  The label that the direction points to is called
the {\em original label\/} and the other label is called the {\em new
  label}.  Assume that a directed gash $g$ points to a puzzle piece
$q$ that contributes the original label of $g$, and that no other
gashes are located on the sides of $q$.  Assume also that there exists
a puzzle piece $q'$ of the same shape as $q$, such that $q'$ has the
new label of $g$ on its side corresponding to $g$, and another label
of $q'$ is equal to the label of $q$ on the same side.  In this case
the gash $g$ can be {\em propagated\/} by replacing $q$ with $q'$.
This replaces the gash $g$ with its new label and creates a new gash
on a different side of $q'$.  The following are examples of
propagations.
\[
\pic{.75}{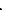} \raisebox{6mm}{$\ \mapsto\ $} \pic{.75}{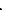} 
\hspace{20mm}
\pic{.75}{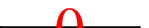} \raisebox{6mm}{$\ \mapsto\ $} \pic{.75}{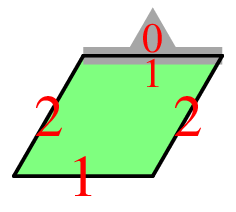}
\]
If the puzzle piece $q$ is equivariant, then the only possible way to
propagate $g$ is to move this gash to the opposite side of $q$; this
follows because opposite sides of any equivariant piece have the same
label.  On the other hand, if $q$ is a triangular puzzle piece, then
the following lemma implies that $g$ can be propagated in at most one
way.

\begin{lemma}\label{lem:onebad}
  Let $a,b,c,x,y,z$ be labels such that $a \neq x$, $b \neq y$, and $c
  \neq z$.  Then at least one of the following triangles is not a
  valid puzzle piece.\vspace{-2mm}
  \[ 
  p = 
  \raisebox{-4mm}{
    \psfrag{a}{$x$}
    \psfrag{e}{$b$}
    \psfrag{c}{$c$}
    \pic{.7}{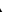} 
  }
  \ \ \ ; \ \ \ 
  q = 
  \raisebox{-4mm}{
    \psfrag{a}{$a$}
    \psfrag{e}{$y$}
    \psfrag{c}{$c$}
    \pic{.7}{uaec2} 
  }
  \ \ \ ; \ \ \ 
  r = 
  \raisebox{-4mm}{
    \psfrag{a}{$a$}
    \psfrag{e}{$b$}
    \psfrag{c}{$z$}
    \pic{.7}{uaec2} 
  }
  \]
\end{lemma}
\begin{proof}
  Since there are finitely many puzzle pieces, this lemma can be
  checked case by case.  However, the lemma is also true with the more
  general definition of puzzle pieces that Knutson gave in
  \cite{knutson:conjectural}.  We will prove the lemma in this
  generality.  In this proof we will therefore use the definition of
  puzzle pieces from \cite{knutson:conjectural}, which can be stated
  as follows.  Each $x \in \N$ is a label and we set $\min(x) =
  \max(x) = x$.  Whenever $a$ and $b$ are labels such that $\max(a) <
  \min(b)$, we declare that $c = (b,a)$ is also a label and set
  $\min(c)=\min(a)$ and $\max(c)=\max(b)$.  A triangular puzzle piece
  is any small triangle of the form
  \[
  {
    \psfrag{a}{$x$}
    \psfrag{e}{$x$}
    \psfrag{c}{$x$}
    \pic{.7}{uaec2}
  }
  \ \ \ \ \ \raisebox{5mm}{\text{ or }} \ \ \ \ \ 
  {
    \psfrag{a}{$a$}
    \psfrag{e}{$b$}
    \psfrag{c}{$c$}
    \pic{.7}{uaec2}
  }
  \]
  where $x \in \N$ and $c = (b,a)$ is a label.  For labels $a$ and $b$
  we will write $a < b$ if $\max(a) < \min(b)$.  The {\em depth\/} of
  a label is its depth as a rooted binary tree.

  Now assume that $a, b, c, x, y, z$ are labels in this sense and the
  triangles $p$, $q$, $r$ of the lemma are puzzle pieces.  If $z=a=b
  \in \N$, then we have either $y=(a,c)$ and $c<a=b$, or
  $c=(y,a)=(y,b)$.  In both cases the triangle $p$ is not a puzzle
  piece.  We may therefore assume that all three triangles are
  composed.

  We claim that exactly one of the identities $x=(c,b)$, $y=(a,c)$,
  $z=(b,a)$ is true.  If two of the identities are true, say $x=(c,b)$
  and $y=(a,c)$, then $b<c<a$ implies that $r$ is not a puzzle piece.
  On the other hand, if none of the identities are true, then we may
  assume without loss of generality that $c$ is the deepest of the
  labels $a,b,c$, and we must have $c=(b,x)=(y,a)$, contradicting that
  $a\neq x$.

  By the claim, we may assume that $x\neq(c,b)$ and $y\neq(a,c)$ and
  $z=(b,a)$.  In particular, we have $a<b$.  If $c=(b,x)$, then
  $c\neq(y,a)$, so we must have $a=(c,y)=((b,x),y)$, contradicting
  $a<b$.  It follows that $b=(x,c)$.  Since we have either $c=(y,a)$
  or $a=(c,y)$, we again deduce that $a<b$ is impossible.  This
  completes the proof.
\end{proof}

\subsection{Equivalence classes of gashes}\label{ssec:gashclass}

We will consider a {\em directed gash\/} as an object that exists
independently of its appearance in a puzzle.  In other words, a
directed gash consists of a direction and two labels, but not a
location.  We will use the textual notation $\ds\frac{\,a\,}{\,b\,}$,
$b/a$, and $a\backslash b$ also for directed gashes when the direction
of the gash is clear from the context.  Given directed gashes $g$ and
$h$, we say that $h$ is {\em immediately reachable\/} from $g$ if $h$
can be obtained by propagating $g$ across a single triangular puzzle
piece.  For example, the first propagation displayed in
Section~\ref{ssec:propagate} shows that the gash
\raisebox{-1mm}{\pic{.6}{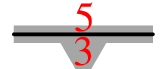}} is immediately reachable from
\raisebox{-3mm}{\pic{.6}{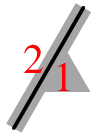}}.  Notice that, if $h$ is obtained
from $g$ by a propagation that replaces a puzzle piece $q$ with
another piece $q'$, then $g$ is obtained from $h$ by a propagation
that replaces the 180 degree rotation of $q$ with the 180 degree
rotation of $q'$.  It follows that `immediately reachable' is a
symmetric relation.

Let $[g]$ denote the set of directed gashes that can be reached from
$g$ by a series of propagations, i.e.\ we have $h \in [g]$ if and only
if there exists a sequence $g = g_0, g_1, \dots, g_k = h$ such that
$g_i$ is immediately reachable from $g_{i-1}$ for each $i$.  The set
$[g]$ is called the {\em class\/} of $g$.  Define the {\em opposite
  gash\/} of $g$ to be the gash $\wh g$ obtained by interchanging the
labels of $g$ and keeping the direction.  For example,
$\raisebox{-1.7mm}{\pic{.6}{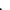}}$ and
$\raisebox{-1.7mm}{\pic{.6}{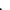}}$ are opposite gashes.  Notice
that $[\wh g] = \{ \wh h \mid h \in [g] \}$.  Similarly, if $g'$ is
obtained by rotating $g$ by some angle, then $[g']$ is obtained from
$[g]$ by rotating all elements by the same angle.  The directed gashes
$g$ and $h$ are said to be in {\em opposite classes\/} if $[\wh g] =
[h]$.  All gash classes that contain at least two gashes are rotations
of one of the following four classes or their opposites.
\[
\begin{split}
\left[ \raisebox{-4mm}{\pic{.7}{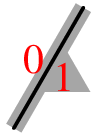}} \right] \ 
&= \ 
\left\{
\raisebox{-2mm}{\pic{.7}{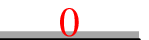}} ,
\raisebox{-4mm}{\pic{.7}{gash110}} ,
\raisebox{-4mm}{\pic{.7}{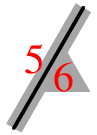}} ,
\raisebox{-4mm}{\pic{.7}{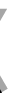}} ,
\raisebox{-4mm}{\pic{.7}{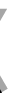}} ,
\raisebox{-2mm}{\pic{.7}{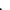}}
\right\}
\\
\left[ \raisebox{-4mm}{\pic{.7}{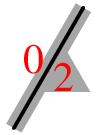}} \right] \ 
&= \ 
\left\{
\raisebox{-2mm}{\pic{.7}{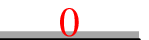}} ,
\raisebox{-2mm}{\pic{.7}{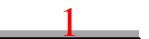}} ,
\raisebox{-4mm}{\pic{.7}{gash120}} ,
\raisebox{-4mm}{\pic{.7}{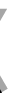}} ,
\raisebox{-4mm}{\pic{.7}{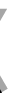}}
\right\}
\\
\left[ \raisebox{-4mm}{\pic{.7}{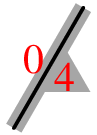}} \right] \ 
&= \ 
\left\{
\raisebox{-2mm}{\pic{.7}{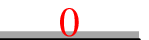}} ,
\raisebox{-4mm}{\pic{.7}{gash140}} ,
\raisebox{-4mm}{\pic{.7}{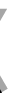}} ,
\raisebox{-2mm}{\pic{.7}{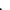}}
\right\}
\\
\left[ \raisebox{-4mm}{\pic{.7}{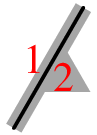}} \right] \ 
&= \ 
\left\{
\raisebox{-4mm}{\pic{.7}{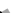}} ,
\raisebox{-2mm}{\pic{.7}{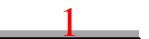}} ,
\raisebox{-2mm}{\pic{.7}{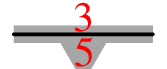}} ,
\raisebox{-4mm}{\pic{.7}{gash121}} ,
\raisebox{-4mm}{\pic{.7}{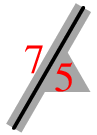}} ,
\raisebox{-4mm}{\pic{.7}{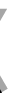}}
\right\}
\end{split}
\]
The directed gashes $g$ for which $[g] = \{g\}$ are the gashes that
can never be propagated.  These gashes are rotations of the following
seven gashes or their opposites.
\[
\pic{.7}{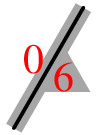} \ \ \ 
\pic{.7}{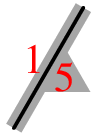} \ \ \ 
\pic{.7}{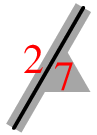} \ \ \ 
\pic{.7}{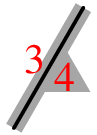} \ \ \ 
\pic{.7}{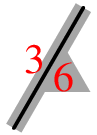} \ \ \ 
\pic{.7}{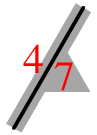} \ \ \ 
\pic{.7}{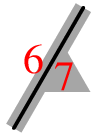}
\]
To see that the displayed gashes account for everything, notice that
none of them are rotations of (opposites of) each other, there are 28
of them, and $28 \cdot 12 = 336$ is the total number of directed
gashes.

\subsection{Flawed puzzles}\label{ssec:flawed}

A {\em flawed puzzle\/} is a puzzle that contains a flaw.  The flaw
can be of three different types: a {\em gash pair\/} on a border
segment, a {\em temporary puzzle piece}, or a {\em marked scab}.  All
types of flaws are represented in the following three puzzles, which
have already been encountered in Section~\ref{ssec:mut_intro}.
\[
\pic{.7}{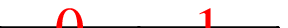} \ \ \ \
\pic{.7}{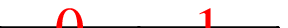} \ \ \ \
\pic{.7}{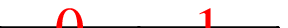}
\]
All boundary labels of a flawed puzzle must be simple.  Any flawed
puzzle has one or more {\em resolutions\/} where the flaw is replaced
with two directed gashes.  These resolutions are used to define the
{\em mutations\/} of the flawed puzzle.  As we will see, the gashes of
a resolution are always in opposite classes.  We proceed to discuss
each type of flaw in more detail.

\subsection{Gash pairs}\label{ssec:gash_pairs}

A {\em gash pair\/} is a pair of gashes located on a single border
segment of a puzzle.  If the puzzle is rotated so that the gashed
border segment is at the top of the puzzle, then the segment of edges
between the two gashed edges should have one of the following three
forms:
\[
\pic{.7}{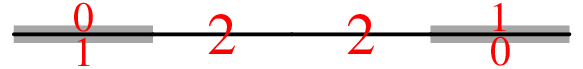} \hmm{5} \pic{.7}{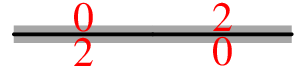} \hmm{5} \pic{.7}{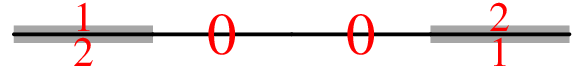}
\]
In the first and third forms, the middle segment may consist of any
number of edges with the indicated labels, including zero.  Notice
that if $u$ is the sequence of labels on or above the border segment,
and $u'$ is the sequence of labels on or below the segment, then we
have $u \to u'$ with the notation of Section~\ref{sec:recurse}.

The gashes of a gash pair should be considered as directed towards the
interior of the puzzle.  A flawed puzzle containing a gash pair is
therefore its own resolution.  However, we usually omit the direction
of gash pairs in pictures.  Notice also that the gashes of a gash pair
are opposite to each other.

\subsection{Temporary puzzle pieces}\label{ssec:temporary}

According to Definition~\ref{def:temporary} below, a {\em temporary
  puzzle piece\/} is a small triangle from the following list.
Temporary puzzle pieces are colored yellow and may be rotated.
\[
\pic{.7}{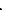} \hmm{5} \pic{.7}{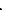} \hmm{5} \pic{.7}{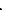} \hmm{5}
\pic{.7}{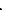} \hmm{5} \pic{.7}{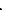} \hmm{5} \pic{.7}{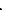}
\]
A flawed puzzle containing a temporary piece is the same as a puzzle,
except that exactly one temporary puzzle piece is used together with
the valid puzzle pieces from Section~\ref{sec:equiv}.  The following
formal definition of temporary puzzle pieces and their resolutions is
valid also for three-step puzzles, see \cite{buch:3step}.

\begin{defn}\label{def:temporary}
  Let $x$, $y$, and $z$ be puzzle labels.  The triangle 
  \[
  \psfrag{a}{$x$}
  \psfrag{e}{$y$}
  \psfrag{c}{$z$}
  t \ = \ \ \ \raisebox{-5.7mm}{\pic{.7}{uaec2}} \ .
  \]
  with these labels is a {\em temporary puzzle piece\/} if and only if
  there exist puzzle labels $x', x'', y', y'', z', z''$ such that all
  of the following triangles are valid puzzle pieces:
  \[
  {
    \psfrag{a}{$x$}
    \psfrag{e}{$y'$}
    \psfrag{c}{$\!z''$}
    r_1 = \raisebox{-5.7mm}{\pic{.7}{uaec2}}
  } \ \ \ \ 
  {
    \psfrag{a}{$x'$}
    \psfrag{e}{$y''$}
    \psfrag{c}{$z$}
    r_2 = \raisebox{-5.7mm}{\pic{.7}{uaec2}}
  } \ \ \ \ 
  {
    \psfrag{a}{$x''$}
    \psfrag{e}{$y$}
    \psfrag{c}{$\!z'$}
    r_3 = \raisebox{-5.7mm}{\pic{.7}{uaec2}}
  } \ \ \ \ 
  {
    \psfrag{a}{$x'$}
    \psfrag{e}{$y'$}
    \psfrag{c}{$\!z'$}
    t' = \raisebox{-5.7mm}{\pic{.7}{uaec2}}
  } \ \ \ \ 
  {
    \psfrag{a}{$x''$}
    \psfrag{e}{$y''$}
    \psfrag{c}{$\!z''$}
    t'' = \raisebox{-5.7mm}{\pic{.7}{uaec2}}
  }
  \]
  In this case the {\em resolutions\/} of $t$ are obtained by
  replacing two of the sides of $t$ with gashes directed away from
  $t$, such that the original labels come from $t$ and the new labels
  come from $r_1$, $r_2$, or $r_3$:
  \[
  {
    \psfrag{x}{$x$}
    \psfrag{y}{$y$}
    \psfrag{yp}{$\,\,y'$}
    \psfrag{z}{$z$}
    \psfrag{zp}{$\,z''$}
    \pic{.9}{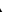}
  }
  \hspace{15mm}
  {
    \psfrag{x}{$x$}
    \psfrag{xp}{$x'$}
    \psfrag{y}{$y$}
    \psfrag{yp}{$y''$}
    \psfrag{z}{$z$}
    \pic{.9}{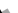}
  }
  \hspace{15mm}
  {
    \psfrag{x}{$x$}
    \psfrag{xp}{$\!\!x''$}
    \psfrag{y}{$y$}
    \psfrag{z}{$z$}
    \psfrag{zp}{$\,z'$}
    \pic{.9}{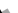}
  }
  \]
\end{defn}

We need the following properties and classification of the resolutions
of temporary puzzle pieces.

\begin{prop}\label{prop:illres}
  {\rm(a)} Let $x, y, z, x', y''$ be puzzle labels.  The gashed
  triangle
  \[
  \raisebox{9mm}{$\wtil t \ = \ \ \ $}
  \psfrag{x}{$x$}
  \psfrag{xp}{$x'$}
  \psfrag{y}{$y$}
  \psfrag{yp}{$y''$}
  \psfrag{z}{$z$}
  \pic{.9}{illres3}\vspace{-5mm}
  \]
  is a resolution of a temporary puzzle piece if and only
  if its two gashes are in opposite classes and the triangle $r_2$ of
  Definition~\ref{def:temporary} is a valid puzzle piece.

  {\rm(b)} Each temporary puzzle piece $t$ has exactly three
  resolutions.  In other words, the valid puzzle pieces $r_1$, $r_2$,
  $r_3$ of Definition~\ref{def:temporary} are uniquely determined by
  $t$.
\end{prop}
\begin{proof}
  Assume first that $\wtil t$ is a resolution of a temporary puzzle
  piece $t$, and let $z'$, $z''$, $r_1$, $r_3$, $t'$, and $t''$ be as
  in Definition~\ref{def:temporary}.  Let $g_1$ be the left directed
  gash of $\wtil t$ and let $g_2$ be the right gash.  The valid puzzle
  pieces $r_3$ and $t''$ then show that $[g_1] =
  [\ds\frac{\,z'\,}{\,z''\,}]$, while $r_1$ and $t'$ show that $[g_2]
  = [\ds\frac{\,z''\,}{\,z'\,}]$, with both horizontal gashes directed
  towards the north.  This shows that $g_1$ and $g_2$ are in opposite
  classes.
  
  To establish the rest of the proposition, one first checks that each
  triangle $t$ in the left column of Table~\ref{ssec:temporary} is a
  temporary puzzle piece.  In fact, if $x, y, z$ are the labels of $t$,
  and we let $x', x'', y', y'', z', z''$ be the unique labels such
  that $y',z'' \leq x$\,,\, $z',x'' \leq y$\,,\, $x',y'' \leq z$, and
  the triangles $r_1, r_2, r_3$ of Definition~\ref{def:temporary} are
  valid puzzle pieces, then $t'$ and $t''$ are also valid puzzle
  pieces.  This shows that $t$ is a temporary puzzle piece, and also
  that the three gashed triangles next to $t$ in
  Table~\ref{ssec:temporary} are resolutions of $t$.  On the other hand,
  by inspection of the gash classes displayed in
  Section~\ref{ssec:gashclass} it is easy to check that, up to
  rotation, all gashed triangles that satisfy the condition in part
  (a) are represented in the right column of Table~\ref{ssec:temporary}.
  The proposition follows from this.
\end{proof}

\begin{table}
  \begin{tabular}{|c|ccc|}
    \hline
    \multicolumn{4}{|l|}{
      \ {\sc Table \ref{ssec:temporary}.} \
      Temporary puzzle pieces and their resolutions.
      \raisebox{4.5mm}{\mbox{}}\raisebox{-2.5mm}{\mbox{}}
    }
    \\
    \hline
    \vspace{-2.5mm}\mbox{\hspace{19mm}}%
    &\mbox{\hspace{19mm}}&\hbox{\hspace{19mm}}&\\
    \pic{.7}{u333} &
    \ \ \ \pic{.7}{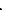} &
    \pic{.7}{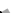} &
    \pic{.7}{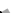} \ \ \ 
    \\
    \vspace{-2.5mm}&&&\\
    \hline
    \vspace{-2.5mm}&&&\\
    \pic{.7}{u444} &
    \ \ \ \pic{.7}{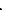} &
    \pic{.7}{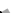} &
    \pic{.7}{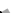} \ \ \ 
    \\
    \vspace{-2.5mm}&&&\\
    \hline
    \vspace{-2.5mm}&&&\\
    \pic{.7}{u555} &
    \ \ \ \pic{.7}{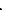} &
    \pic{.7}{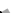} &
    \pic{.7}{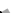} \ \ \ 
    \\
    \vspace{-2.5mm}&&&\\
    \hline
    \vspace{-2.5mm}&&&\\
    \pic{.7}{u645} & 
    \ \ \ \pic{.7}{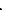} & 
    \pic{.7}{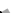} & 
    \pic{.7}{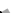} \ \ \ 
    \\
    \vspace{-2.5mm}&&&\\
    \hline
    \vspace{-2.5mm}&&&\\
    \pic{.7}{u753} & 
    \ \ \ \pic{.7}{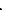} & 
    \pic{.7}{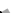} & 
    \pic{.7}{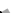} \ \ \ 
    \\
    \vspace{-2.5mm}&&&\\
    \hline
    \vspace{-2.5mm}&&&\\
    \pic{.7}{u176} &
    \ \ \ \pic{.7}{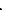} &
    \pic{.7}{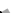} &
    \pic{.7}{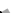} \ \ \ 
    \vspace{-2.5mm}\\ 
    &&& \\
    \hline
  \end{tabular}
\end{table}

\begin{remark}
  Given a temporary puzzle piece $t$, the valid puzzle pieces used to
  form the resolutions of $t$ are obtained by keeping one label $z$ of
  $t$ and replacing the two other labels with the unique integers
  smaller than or equal to $z$ such that the resulting triangle is a
  valid puzzle piece.  This is a coincidence that holds for two-step
  puzzles but not for three-step puzzles \cite{buch:3step}.
\end{remark}

\subsection{Scabs}\label{ssec:scabs}

A {\em scab\/} means a small rhombus consisting of two triangular
puzzle pieces with matching labels next to each other, so that the
rhombus is not invariant under 180 degree rotation.  In other words,
the triangular puzzle pieces are not rotations of each other.  Any
puzzle containing one or more scabs can be turned into a flawed puzzle
by marking one of the scabs.  Marked scabs are colored light blue in
pictures.

Let $s$ be a scab and assume that $q$ is an equivariant puzzle piece
of the same shape as $s$, such that two sides of $s$ and $q$ share the
same labels.  In this case the two labels that $s$ and $q$ agree about
must be on sides connected by an obtuse angle.  The gashed rhombus
$\wtil s$ resulting from replacing $s$ with $q$ is then called a {\em
  resolution\/} of $s$.  More precisely, $\wtil s$ is obtained from
$q$ by replacing two of its sides with gashes directed away from $q$.
These are the sides where the labels of $q$ and $s$ disagree, and the
original labels of the gashes come from $s$ while the new labels come
from $q$.  The following is an example.
\[
s = \raisebox{-8.5mm}{\pic{.7}{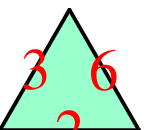}} \ \ ; \hmm{10}
q = \raisebox{-8.5mm}{\pic{.7}{uf34}} \ \ ; \hmm{10}
\wtil s = \raisebox{-8.5mm}{\pic{.7}{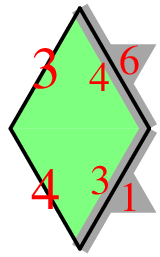}}
\]

\begin{prop}\label{prop:scabres}
  {\rm(a)} Let $x, x', y, y'$ be puzzle labels.  The gashed rhombus
  \[
  \psfrag{x}{$x$}
  \psfrag{y}{$y$}
  \psfrag{xc}{$x'$}
  \psfrag{yc}{$y'$}
  \pic{.7}{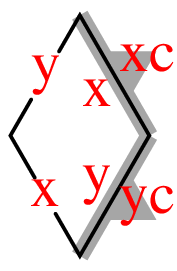}
  \]
  is a resolution of a scab if and only if the inner labels form an
  equivariant puzzle piece and the two gashes are in opposite classes.

  {\rm(b)} Each scab has exactly one resolution.
\end{prop}
\begin{proof}
  Assume that the gashed rhombus is a resolution of a scab, and let
  $z$ be the label of the middle edge in this scab.  Then the valid
  puzzle pieces
  \[
  {
    \psfrag{a}{$x'$}
    \psfrag{e}{$y$}
    \psfrag{c}{$z$}
    \pic{.7}{uaec2}
  }
  \ \ \ \raisebox{4mm}{\text{and}} \ \ \ 
  {
    \psfrag{a}{$x$}
    \psfrag{e}{$y'$}
    \psfrag{c}{$z$}
    \pic{.7}{uaec2}
  }
  \]
  show that the two gashes of the resolution are in opposite classes.
  On the other hand, by inspection of the gash classes displayed in
  Section~\ref{ssec:gashclass} it is easy to check that, up to
  rotation, all gashed rhombuses that satisfy the condition in part
  (a) are represented in Table~\ref{ssec:scabs}.  Since
  Table~\ref{ssec:scabs} also documents that every scab has a
  resolution, this completes the proof.
\end{proof}

\begin{table}
  \begin{tabular}{|l|}
    \hline
    \ {\sc Table~\ref{ssec:scabs}.} \ Scabs and their resolutions.
    \raisebox{4.5mm}{\mbox{}}\raisebox{-2.5mm}{\mbox{}}
    \\
    \hline
    \vspace{-2.5mm}\\
    \pic{.7}{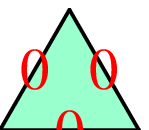} \pic{.7}{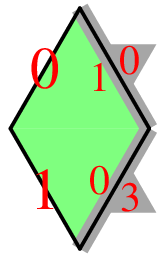} \ \ \ \ 
    \pic{.7}{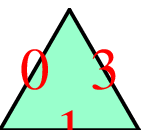} \pic{.7}{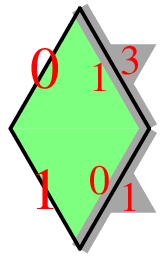} \ \ \ \ 
    \pic{.7}{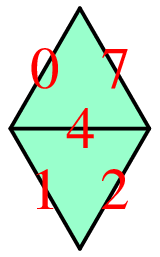} \pic{.7}{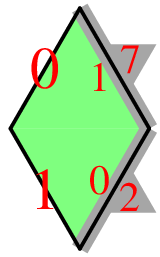} \ \ \ \ 
    \pic{.7}{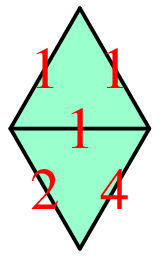} \pic{.7}{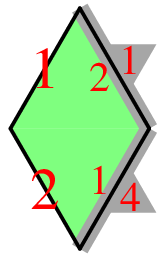}
    \\
    \pic{.7}{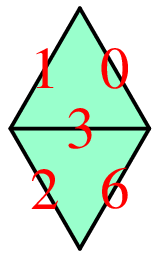} \pic{.7}{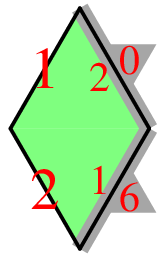} \ \ \ \ 
    \pic{.7}{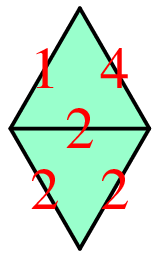} \pic{.7}{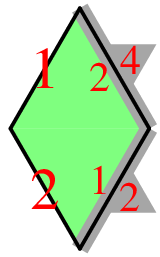} \ \ \ \ 
    \pic{.7}{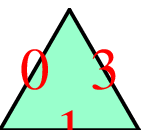} \pic{.7}{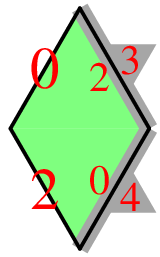} \ \ \ \ 
    \pic{.7}{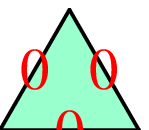} \pic{.7}{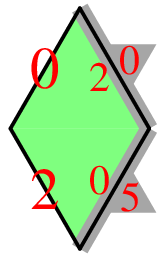} 
    \\
    \pic{.7}{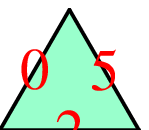} \pic{.7}{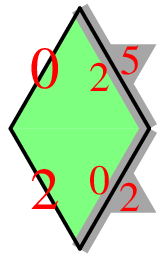} \ \ \ \ 
    \pic{.7}{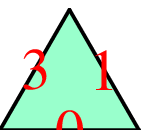} \pic{.7}{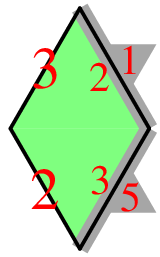} \ \ \ \ 
    \pic{.7}{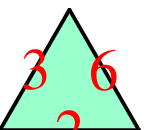} \pic{.7}{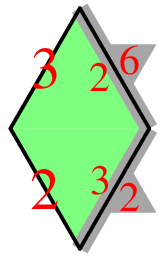} \ \ \ \ 
    \pic{.7}{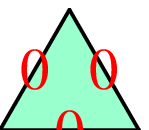} \pic{.7}{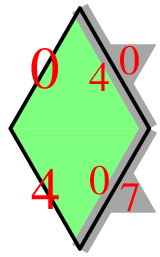} 
    \\
    \pic{.7}{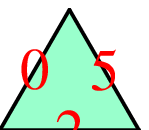} \pic{.7}{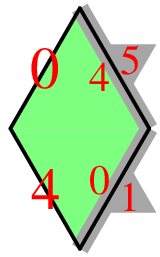} \ \ \ \ 
    \pic{.7}{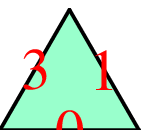} \pic{.7}{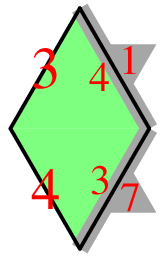} \ \ \ \ 
    \pic{.7}{s3614} \pic{.7}{s3614res} \ \ \ \ 
    \pic{.7}{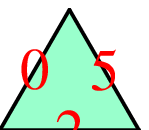} \pic{.7}{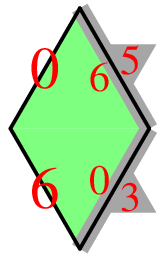}
    \\
    \pic{.7}{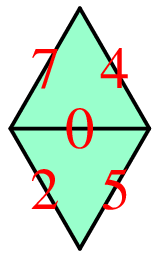} \pic{.7}{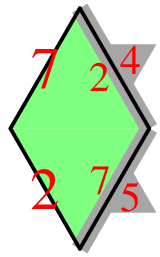} 
    \\
    \hline
  \end{tabular}
\end{table}

\subsection{Mutations}\label{ssec:mutations}

Let $P$ be a flawed puzzle and let $\wtil P$ be the result of
replacing the flaw in $P$ with one of its resolutions.  Then $\wtil P$
is a gashed puzzle called a {\em resolution\/} of $P$.  The two
directed gashes in $\wtil P$ are in opposite classes, and these gashes
are either connected or separated by a sequence of edges from the same
border segment.  The {\em right gash\/} of $\wtil P$ is the rightmost
of the two gashes for an observer standing between the gashes and
facing the direction of the gashes.  The other gash in $\wtil P$ is
called the {\em left gash}.  The following gashed puzzles are
resolutions of the flawed puzzles displayed in
Section~\ref{ssec:flawed}.  The right gashes of these puzzles are
$0\backslash 2$, $\ds\frac{\,0\,}{\,7\,}$, and $4/0$.
\[
\pic{.7}{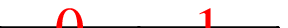} \ \ \ \
\pic{.7}{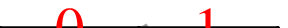} \ \ \ \
\pic{.7}{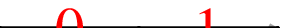}
\]

Define the {\em propagation path\/} of a gash in $\wtil P$ to be the
sequence of edges that change if we repeatedly propagate the gash
until no more propagations are possible.  We let $\Phi(\wtil P)$
denote the result of propagating both gashes in $\wtil P$ as far as
possible and then reversing the directions of the gashes.  This is
well defined by the first claim in the following result.

\begin{thm}\label{thm:mutate}
  Let $\wtil P$ be a resolution of a flawed puzzle.  Then the
  propagation paths of the two gashes in $\wtil P$ are disjoint.
  Furthermore, $\Phi(\wtil P)$ is a resolution of a unique flawed
  puzzle, and we have $\Phi(\Phi(\wtil P)) = \wtil P$.
\end{thm}

Theorem~\ref{thm:mutate} will be proved in Section~\ref{ssec:mutproof}.
We will say that two flawed puzzles $P$ and $Q$ are {\em mutations\/}
of each other if $P$ has a resolution $\wtil P$ such that $\Phi(\wtil
P)$ is a resolution of $Q$.  The set of all flawed puzzles can be
arranged in a {\em mutation graph\/}, where each flawed puzzle is
connected to its mutations.  Figure~\ref{ssec:mutations} shows one
component of this graph.  For each flawed puzzle in the figure we have
also indicated the set of edges that are changed by at least one
mutation.

It should be noted that resolutions of flawed puzzles and propagation
of gashes commute with rotations and dualization.  This simplifies our
proof of Theorem~\ref{thm:mutate}, and it implies that mutation
commutes with rotations and dualization.

\begin{table}
  \begin{tabular}{|l|}
    \hline
    \ {\sc Figure~\ref{ssec:mutations}.} \ A connected component of the
    mutation graph.
    \raisebox{4.5mm}{\mbox{}}\raisebox{-2.5mm}{\mbox{}}
    \\
    \hline
    \vspace{-2.5mm}\\
    \noindent\pic{.49}{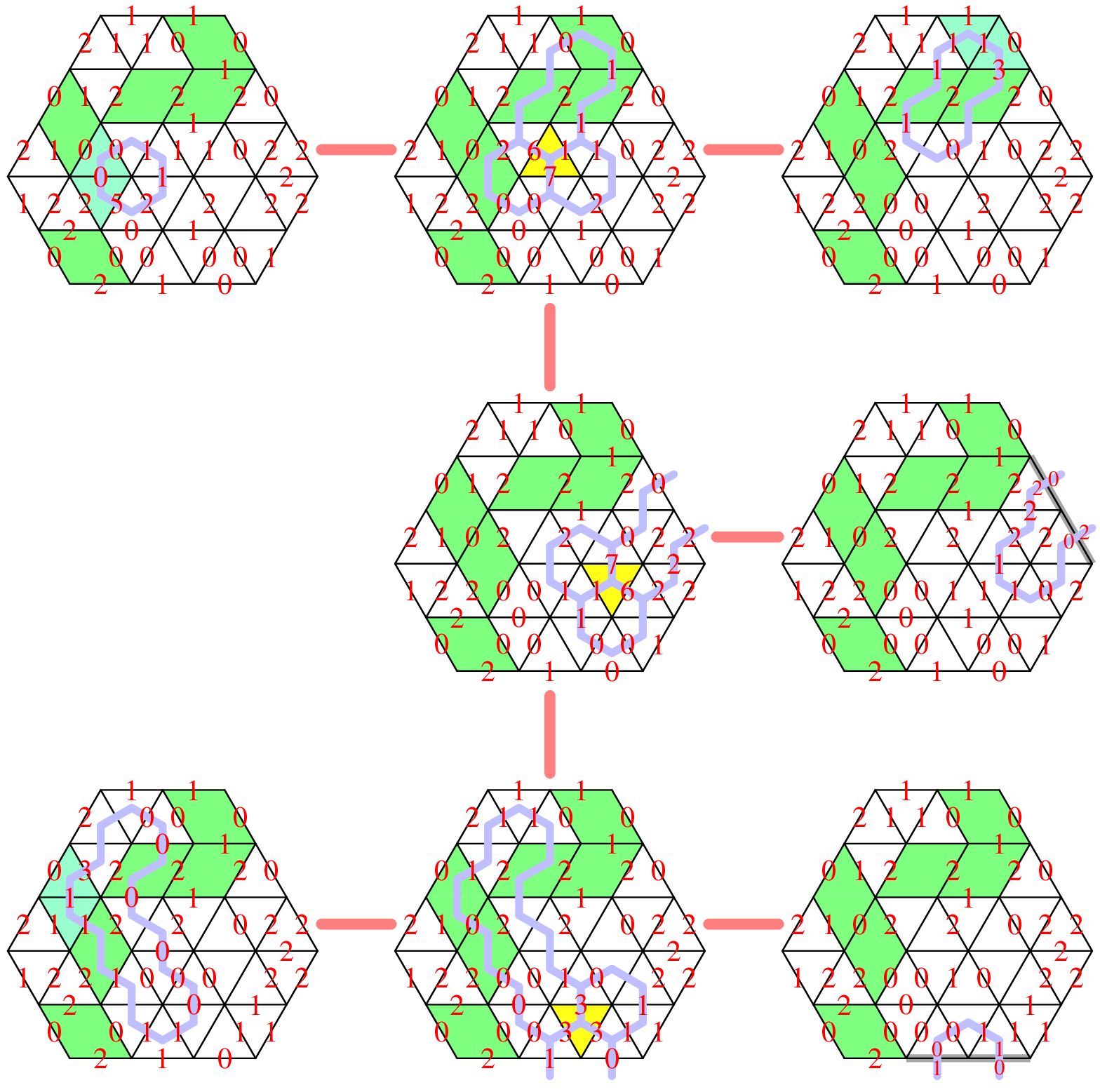}
    \\
    \hline
  \end{tabular}
\end{table}

\begin{example}
  In early versions of this paper we conjectured that every connected
  component of the mutation graph is a tree.  However, the reader may
  check that the following puzzle belongs to a cycle of length 13 in
  its component.  Notice also that since this puzzle is dual to
  itself, dualization of puzzles provides an involution of this
  component.
  \begin{center}
    \pic{.6}{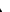}
  \end{center}
\end{example}

\subsection{Proof of Theorem~\ref{thm:mutate}}\label{ssec:mutproof}

Let $\wtil P$ be a resolution of a flawed puzzle.  After rotating and
possibly dualizing this puzzle, we may assume that the right gash in
$\wtil P$ is equivalent to one of the following directed gashes:
\[
\raisebox{-6mm}{\pic{.7}{gash110}} \ \ \ \ \ 
\raisebox{-6mm}{\pic{.7}{gash120}} \ \ \ \ \ 
\raisebox{-6mm}{\pic{.7}{gash140}}
\]
We first assume that the right gash is in the equivalence class
\[
[0/1] \ = \ 
\left[ \raisebox{-4mm}{\pic{.7}{gash110}} \right] \ 
= \ 
\left\{
\raisebox{-2mm}{\pic{.7}{gash030}} ,
\raisebox{-4mm}{\pic{.7}{gash110}} ,
\raisebox{-4mm}{\pic{.7}{gash165}} ,
\raisebox{-4mm}{\pic{.7}{gash213}} ,
\raisebox{-4mm}{\pic{.7}{gash245}} ,
\raisebox{-2mm}{\pic{.7}{gash346}}
\right\} \,.
\]
The left and right gashes in $\wtil P$ are connected by a node or a
sequence of edges.  In the latter case, these edges have the label 2
and are located on the north-west border segment of $\wtil P$.
Consider the set of all edges in $\wtil P$ that come from the
following list (with the indicated orientations).  These edges can
also be found in the center of Figure~\ref{ssec:mutproof}(a).
\[
\pic{.7}{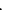} \ \ \ 
\pic{.7}{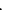} \ \ \ 
\raisebox{2.5mm}{\pic{.7}{}} \ \ \ 
\raisebox{2.5mm}{\pic{.7}{}} \ \ \ 
\raisebox{2.5mm}{\pic{.7}{}} \ \ \ 
\raisebox{2.5mm}{\pic{.7}{}} \ \ \ 
\pic{.7}{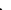} \ \ \ 
\pic{.7}{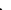}
\]
Let $I$ be the connected component in this set of edges that includes
the node or edges connecting the left and right gashes.  The edges of
$\wtil P$ that are connected to $I$ but not contained in $I$ will be
called the {\em spikes\/} of $I$.  In particular, the left and right
gashes of $P$ are spikes of $I$.  In the following two examples the
edges of $I$ have been colored light blue while the spikes have been
made thick.  The nodes where the spikes are connected to $I$ are drawn
as fat dots.  One can show that no edge outside $I$ can be connected
to $I$ in both ends, but we will not rely on this fact.  In such a
situation the edge would count as two spikes.
\[
\pic{.7}{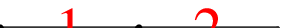} 
\ \ \ \ \ 
\pic{.7}{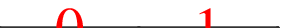}
\]

Let $s_0, s_1, \dots, s_\ell$ be the sequence of spikes obtained when
we start with the right gash and follow the boundary of $I$ in counter
clockwise direction.  Then $s_0$ is the right gash of $\wtil P$ and
$s_\ell$ is the left gash.  Any pair of consecutive spikes
$s_{k-1},s_k$ in the sequence is separated either by a single puzzle
piece or by the boundary of $\wtil P$; the latter happens when part of
the boundary of $I$ is contained in the boundary of $\wtil P$.

Let $\theta_0 \in (0,2\pi]$ be the direction of the first spike $s_0$
in $\wtil P$.  Then choose angles $\theta_1, \dots, \theta_\ell \in
\R$ for the other spikes relative to $\theta_0$.  More precisely, if
$\theta_0, \dots, \theta_{k-1}$ have been chosen, then let $\theta_k$
be the result of adding or subtracting an amount to $\theta_{k-1}$
that represents the change in direction from $s_{k-1}$ to $s_k$.  For
example, in the hypothetical situation
\[
\pic{.5}{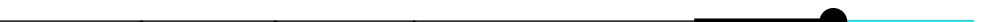}
\]
we have $\ell = 20$ and
\[
\begin{split}
  &(\theta_0,\theta_1,\dots,\theta_{20}) = \\
  &\textstyle (\frac{4}{3}\pi, \pi, \frac{4}{3}\pi, \frac{4}{3}\pi,
  \pi, 2\pi, \frac{5}{3}\pi, \frac{5}{3}\pi, \frac{5}{3}\pi,
  \frac{4}{3}\pi, 3\pi, 3\pi, \frac{8}{3}\pi, \frac{7}{3}\pi,
  \frac{8}{3}\pi, \frac{7}{3}\pi, \frac{7}{3}\pi, \frac{7}{3}\pi,
  \frac{8}{3}\pi, 3\pi, \frac{8}{3}\pi ) \,.
\end{split}
\]

To each spike $s_k$ we now define an {\em adjusted angle\/}
$\wh\theta_k$ that is obtained by subtracting an amount from
$\theta_k$ that depends on both the label of $s_k$ and ($\theta_k$ mod
$2\pi$).  For $s_0$ and $s_\ell$ we use the original labels of the
corresponding gashes.  Figure~\ref{ssec:mutproof}(a) shows all possible
spikes of $I$ together with the amount that should be subtracted in
each case.  Notice that many edges in the figure are used to represent
several spikes with different labels, which is done by listing the
relevant labels.  For example, if $\theta_k = \frac{11}{3}\pi$ and
$s_k$ has label $6$, then we obtain $\wh\theta_k = \frac{5}{3}\pi$,
since the amount $2\pi$ must be subtracted from the angle of any spike
of the form $\raisebox{-2.8mm}{\pic{.5}{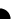}}$.

\input{fig01}

Figure~\ref{ssec:mutproof}(a) separates the collection of possible
spikes to $I$ into the six groups $G_0$, $G_1$, $G_2$, $G_3$, $G_4$,
and $G_5$.  Notice that the right gash $s_0$ belongs to $G_0$, while
the left gash $s_\ell$ belongs to $G_1$.  In particular, we have $\pi
\leq \theta_0 \leq 2\pi$ and $\wh\theta_0 = 0$.

\begin{lemma}\label{lem:adj}
  We have $\wh\theta_0 \leq \wh\theta_1 \leq \dots \leq
  \wh\theta_\ell$.  Furthermore, if two consecutive spikes $s_{k-1}$
  and $s_k$ belong to different spike groups, or if $s_{k-1}$ and
  $s_k$ are separated by the boundary of $\wtil P$, then $\wh\theta_k
  - \wh\theta_{k-1} \geq \frac{1}{3}\pi$.
\end{lemma}
\begin{proof}
  Assume first that $s_{k-1}$ and $s_k$ are separated by a puzzle
  piece $q$.  Then the difference $\theta_k-\theta_{k-1}$ is
  determined by $q$.  Since there are finitely many possibilities for
  $q$, the lemma can be checked case by case.

  Table~\ref{ssec:mutproof}(a) lists all possibilities for the puzzle
  piece $q$ when $s_{k-1}$ belongs to $G_0$, $G_1$, or $G_2$.  The
  spikes $s_{k-1}$ and $s_k$ are also identified in each case.  The
  puzzle pieces $q$ for which $s_{k-1}$ is in $G_3$, $G_4$, or $G_5$
  can be obtained by rotating the puzzle pieces in
  Table~\ref{ssec:mutproof}(a) by 180 degrees.  Notice also that the
  puzzle pieces in the table are organized into four rows, depending
  on the exact spike groups that $s_{k-1}$ and $s_k$ belong to.  This
  will be convenient later.

  As an example, if $q =
  \raisebox{-4mm}{\pic{.5}{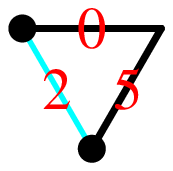}}$, then $s_{k-1} =
  \raisebox{-2.9mm}{\pic{.5}{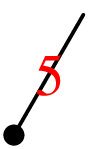}}$ and $s_k =
  \raisebox{-1mm}{\pic{.5}{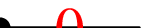}}$ are both in the group $G_1$, and
  we have $\theta_k = \theta_{k-1}-\frac{\pi}{3}$, $\wh\theta_{k-1} =
  \theta_{k-1}-2\pi$, and $\wh\theta_k = \theta_k-\frac{5\pi}{3} =
  \wh\theta_{k-1}$.  On the other hand, if $q =
  \raisebox{-4mm}{\pic{.5}{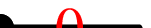}}$, then $s_{k-1} =
  \raisebox{-2.9mm}{\pic{.5}{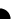}}$ is in the group $G_0$, $s_k =
  \raisebox{-1mm}{\pic{.5}{spike00}}$ is in $G_1$, $\theta_k =
  \theta_{k-1}+\frac{2\pi}{3}$, $\wh\theta_{k-1} =
  \theta_{k-1}-\frac{4\pi}{3}$, and $\wh\theta_k =
  \theta_k-\frac{5\pi}{3} = \wh\theta_{k-1}+\frac{\pi}{3}$.  We leave
  the remaining cases to the reader.

  We finally assume that $s_{k-1}$ and $s_k$ are separated by the
  boundary of $\wtil P$.  Then we have $\theta_k - \theta_{k-1} \geq
  \pi$, and since all boundary labels of $\wtil P$ are simple, it
  follows that $s_{k-1}$ and $s_k$ have simple labels.  Based on these
  observations one may check from Figure~\ref{ssec:mutproof}(a) that
  $\wh\theta_k - \wh\theta_{k-1} \geq \frac{\pi}{3}$, as required.
\end{proof}

\input{sep01}

We first deduce from Lemma~\ref{lem:adj} that our sequence of spikes
goes around the outer boundary of $I$ in counter clockwise direction,
as opposed to going around a hole in $I$ in clockwise direction.  In
other words, situations like the following are impossible.
\[
\pic{.6}{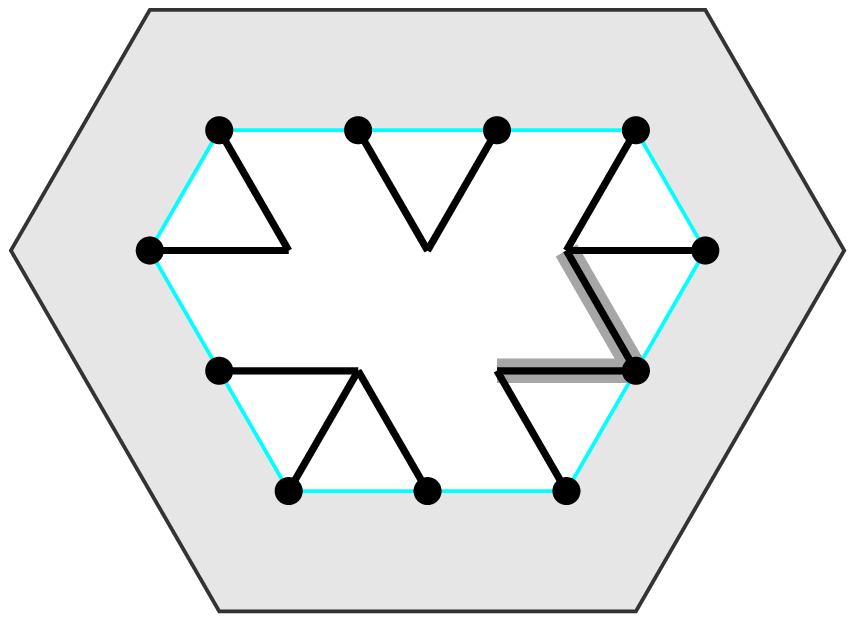}
\]
In fact, if the sequence of spikes went around a hole in $I$, then we
would have $\theta_0-\theta_\ell \in \{ \frac{7\pi}{3}, \frac{8\pi}{3}
\}$, $-\frac{5}{3}\pi \leq \theta_\ell \leq -\pi$, $\wh\theta_0 = 0$,
and $\wh\theta_\ell = -\frac{11}{3}\pi$, which contradicts
Lemma~\ref{lem:adj}.

Since the sequence of spikes goes counter clockwise around the outer
boundary of $I$, we obtain $\theta_\ell-\theta_0 \in \{\pi,
\frac{4\pi}{3}, \frac{5\pi}{3} \}$, $2\pi \leq \theta_\ell \leq 3\pi$,
$\wh\theta_0 = 0$, and $\wh\theta_\ell = \frac{\pi}{3}$.
Lemma~\ref{lem:adj} then implies that for some $r \in [1,\ell]$ we
have $\wh\theta_0 = \wh\theta_1 = \dots = \wh\theta_{r-1} = 0$ and
$\wh\theta_r = \wh\theta_{r+1} = \dots = \wh\theta_\ell =
\frac{\pi}{3}$.  Furthermore, we have $s_k \in G_0$ for $0 \leq k \leq
r-1$ and $s_k \in G_1$ for $r \leq k \leq \ell$.  This implies that
the first $r$ spikes are separated by puzzle pieces from the first row
of Table~\ref{ssec:mutproof}(a), the two middle spikes $s_{r-1}$ and
$s_r$ are separated either by the boundary of $\wtil P$ or by a puzzle
piece from the second row of the table, and the last $\ell-r+1$ spikes
are separated by puzzle pieces from the third row.

When the right gash of $\wtil P$ is propagated, this gash moves
through the spikes $s_k$ for $0 \leq k \leq r-1$.  Each spike $s_k$ is
first replaced with the unique gash in the gash class $[0/1]$ that has
the same orientation as $s_k$ and whose original label is equal to the
label of $s_k$.  Then $s_k$ attains the new label of the same gash,
and the gash moves on.  This follows by observing that the following
substitution of spikes replaces all puzzle pieces in the first row of
Table~\ref{ssec:mutproof}(a) with different valid puzzle pieces.  These
substitutions correspond to the gashes in the gash class $[0/1]$.
\begin{align*}
  \raisebox{-1.3mm}{\pic{.6}{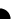}}\ &\mapsto \,
  \raisebox{-1.3mm}{\pic{.6}{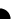}} &
  \raisebox{-3.5mm}{\pic{.6}{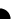}} &\mapsto 
  \raisebox{-3.5mm}{\pic{.6}{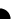}} &
  \raisebox{-3.5mm}{\pic{.6}{spike46}} &\mapsto 
  \raisebox{-3.5mm}{\pic{.6}{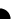}} \\
  \raisebox{-3.5mm}{\pic{.6}{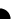}} &\mapsto 
  \raisebox{-3.5mm}{\pic{.6}{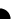}} &
  \raisebox{-3.5mm}{\pic{.6}{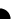}} &\mapsto 
  \raisebox{-3.5mm}{\pic{.6}{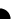}} &
  \raisebox{-1.3mm}{\pic{.6}{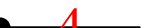}}\, &\mapsto \ 
  \raisebox{-1.3mm}{\pic{.6}{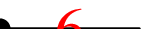}}
\end{align*}
The above propagations will replace the spike $s_{r-1}$ with the
unique gash in the class $[0/1]$ whose orientation and original label
agree with $s_{r-1}$, and this gash points to either the boundary of
$\wtil P$ or a puzzle piece from the second row of
Table~\ref{ssec:mutproof}(a).  An inspection of the puzzle pieces in
this row then shows that the right gash cannot be propagated further.

Similarly, the left gash of $\wtil P$ propagates through the spikes
$s_k$ for $r \leq k \leq \ell$ in reverse order.  Each spike $s_k$ is
first replaced with the unique gash in the opposite gash class $[1/0]$
that has the same orientation as $s_k$ and whose original label is
equal to the label of $s_k$.  Then $s_k$ attains the new label of this
gash, and the gash moves on.  This follows because the substitution of
spikes corresponding to the opposite class $[1/0]$ replaces all puzzle
pieces in the third row of Table~\ref{ssec:mutproof}(a) with different
valid puzzle pieces.  Eventually $s_r$ is replaced with the unique
gash from the opposite class $[1/0]$ with the same orientation and
original label.  At this point an inspection of the second row of
Table~\ref{ssec:mutproof}(a) shows that the left gash cannot be
propagated further (this is also true if the left gash is propagated
before the right gash).

At this point $\Phi(\wtil P)$ is obtained by reversing the directions
of both gashes.  If the spikes $s_{r-1}$ and $s_r$ are both on the
boundary of $\wtil P$, then the (original) labels of these spikes are
simple, and we have $\theta_r-\theta_{r-1} \geq \pi$.  An inspection
of the spike groups $G_0$ and $G_1$ of Figure~\ref{ssec:mutproof}(a)
then shows that $s_{r-1} = \raisebox{-2.9mm}{\pic{.5}{spike41}}$ and
$s_r = \raisebox{-2.9mm}{\pic{.5}{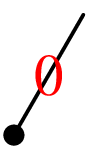}}$.  This implies that
$\Phi(\wtil P)$ is a flawed puzzle with a gash-pair on the south-east
border segment.  Otherwise $s_{r-1}$ and $s_r$ are separated by a
puzzle piece from the second row of Table~\ref{ssec:mutproof}(a), and
this puzzle piece appears in $\Phi(\wtil P)$ with gashes on two sides
that are in opposite classes.  In this case it follows from
Proposition~\ref{prop:illres}(a) or Proposition~\ref{prop:scabres}(a)
that $\Phi(\wtil P)$ is a resolution of a flawed puzzle.

Theorem~\ref{thm:mutate} follows from this when the right gash of
$\wtil P$ is in the gash class $[0/1]$.  The same argument also works
if the right gash is in one of the classes $[0/2]$ or $[0/4]$, except
that Figure~\ref{ssec:mutproof}(a) and Table~\ref{ssec:mutproof}(a)
must be replaced with Figure~\ref{ssec:mutproof}(b) and
Table~\ref{ssec:mutproof}(b) for the class $[0/2]$ and with
Figure~\ref{ssec:mutproof}(c) and Table~\ref{ssec:mutproof}(c) for the
class $[0/4]$.  This completes the proof of Theorem~\ref{thm:mutate}.

\input{fig02}

\input{sep02}

\input{fig04}

\input{sep04}

\input{biject}

%%% Local Variables: 
%%% mode: latex
%%% TeX-master: "puzzle2eq.tex"
%%% End: 

%% file: fig01.tex
\begin{table}
  \begin{tabular}{|l|}
    \hline
    \ {\sc Figure}~\ref{ssec:mutproof}(a). \ Spike groups and
    adjustment angles for the gash class $[0/1]$.
    \raisebox{4.5mm}{\mbox{}}\raisebox{-2.5mm}{\mbox{}}
    \\
    \hline
    \vspace{-2mm}
    \\
    {
      \ \ \ \ \ 
      \psfrag{x}{$\pi$}\psfrag{4x}{$\frac{4\pi}{3}$}
      \psfrag{5x}{$\frac{5\pi}{3}$}\psfrag{6x}{$2\pi$}
      \psfrag{7x}{$\frac{7\pi}{3}$}\psfrag{8x}{$\frac{8\pi}{3}$}
      \psfrag{G0}{$G_0$}\psfrag{G1}{$G_1$}\psfrag{G2}{$G_2$}
      \psfrag{G3}{$G_3$}\psfrag{G4}{$G_4$}\psfrag{G5}{$G_5$}
      \pic{.7}{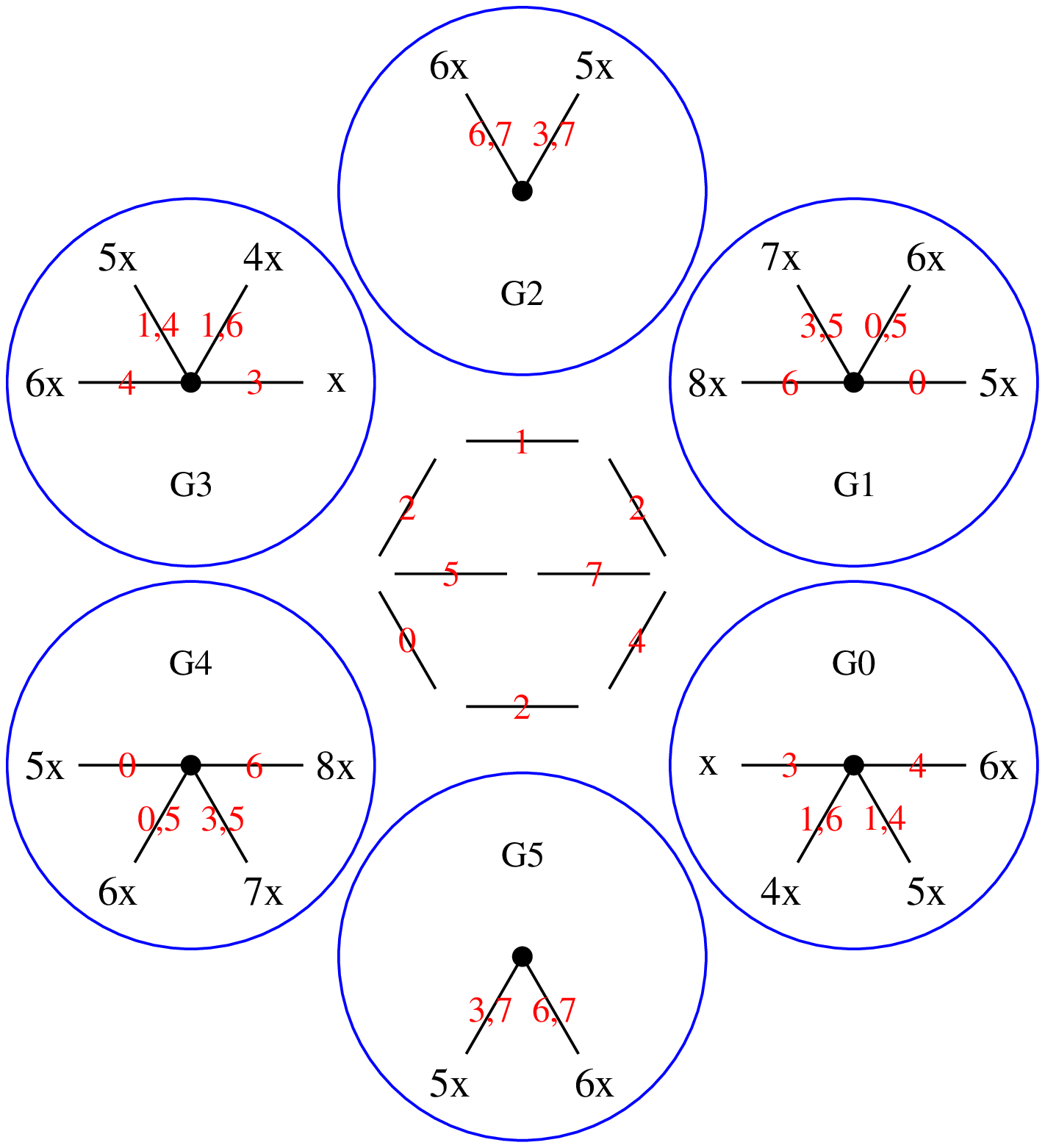}
    }
    \\
    \hline
  \end{tabular}
\end{table}

%%% Local Variables: 
%%% mode: latex
%%% TeX-master: "puzzle2eq.tex"
%%% End: 

%% file: sep01.tex
\begin{table}
  \begin{tabular}{|l|}
    \hline
    \ {\sc Table}~\ref{ssec:mutproof}(a). \ Consecutive spikes for the
    gash class $[0/1]$.
    \raisebox{4.5mm}{\mbox{}}\raisebox{-2.5mm}{\mbox{}}
    \\
    \hline
    \vspace{-3mm}
    \\
    \pic{.6}{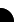}
    \pic{.6}{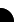}
    \pic{.6}{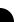}
    \pic{.6}{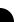}
    \pic{.6}{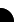}
    \pic{.6}{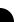}
    \pic{.6}{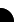}
    \raisebox{-7.5mm}{\pic{.6}{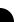}}
    \pic{.6}{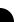}
    \vspace{-6mm}
    \\
    \pic{.6}{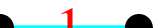}
    \pic{.6}{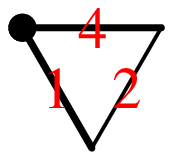}
    \pic{.6}{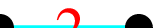}
    \pic{.6}{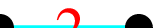}
    \pic{.6}{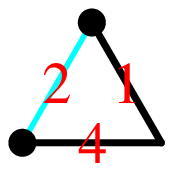}
    \pic{.6}{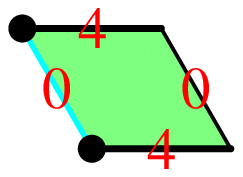}
    \\
    \hline
    \pic{.6}{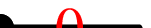}
    \pic{.6}{spike4600-0-6-14.eps}
    \pic{.6}{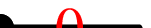}
    \pic{.6}{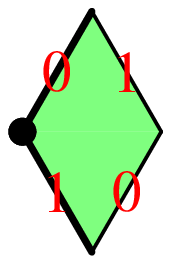}
    \pic{.6}{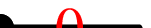}
    \pic{.6}{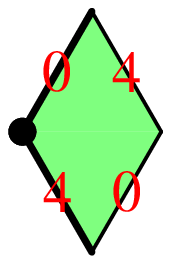}
    \pic{.6}{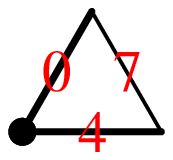}
    \pic{.6}{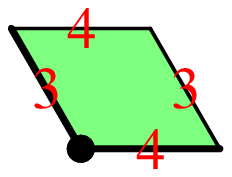}
    \\
    \hline
    \vspace{-3mm}
    \\
    \pic{.6}{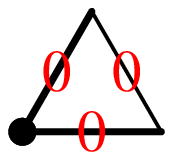}
    \pic{.6}{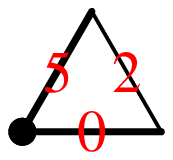}
    \pic{.6}{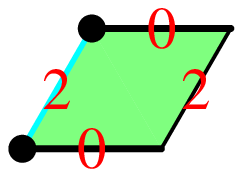}
    \pic{.6}{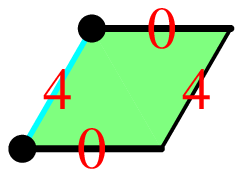}
    \pic{.6}{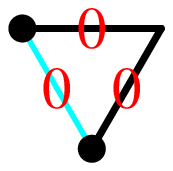}
    \pic{.6}{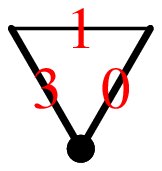}
    \pic{.6}{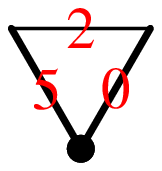}
    \raisebox{-7.5mm}{\pic{.6}{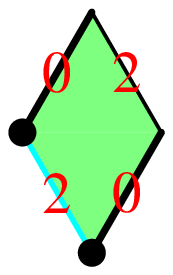}}
    \pic{.6}{spike1500-0-5-2.eps}
    \vspace{-6mm}
    \\
    \pic{.6}{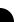}
    \pic{.6}{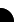}
    \pic{.6}{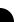}
    \pic{.6}{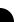}
    \pic{.6}{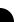}
    \pic{.6}{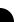}
    \\
    \hline
    \pic{.6}{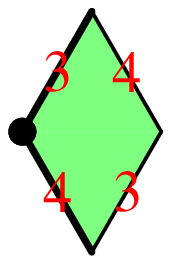}
    \pic{.6}{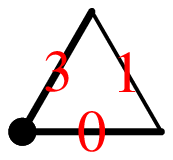}
    \pic{.6}{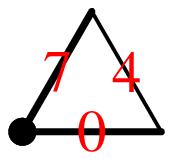}
    \pic{.6}{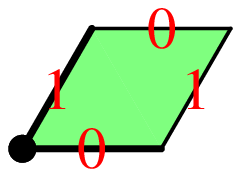}
    \pic{.6}{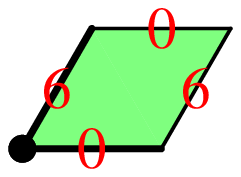}
    \pic{.6}{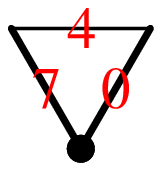}
    \pic{.6}{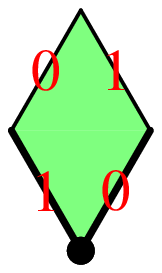}
    \pic{.6}{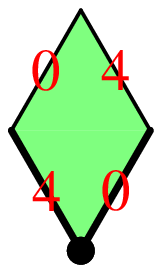}
    \pic{.6}{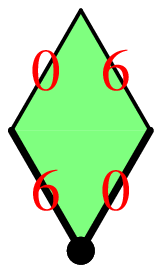}
    \\
    \pic{.6}{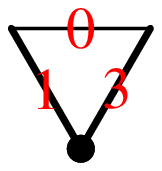} \ 
    \pic{.6}{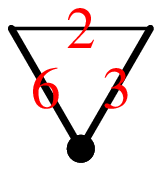} \ \
    \raisebox{-7.5mm}{\pic{.6}{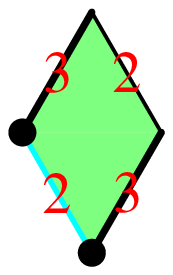}}
    \raisebox{-7.5mm}{\pic{.6}{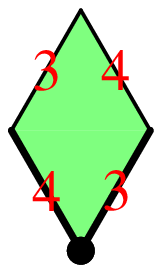}}
    \pic{.6}{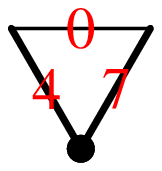}
    \raisebox{-7.5mm}{\pic{.6}{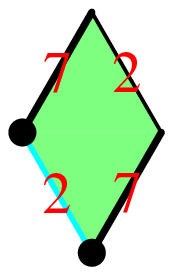}}
    \pic{.6}{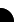}
    \pic{.6}{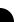}
    \raisebox{-7.5mm}{\pic{.6}{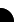}}
    \vspace{-6mm}
    \\
    \pic{.6}{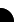}
    \pic{.6}{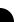}
    \\
    \hline
  \end{tabular}
\end{table}

%%% Local Variables: 
%%% mode: latex
%%% TeX-master: "puzzle2eq.tex"
%%% End: 

%% file: fig02.tex
\begin{table}
  \begin{tabular}{|l|}
    \hline
    \ {\sc Figure}~\ref{ssec:mutproof}(b). \ Spike groups and
    adjustment angles for the gash class $[0/2]$.
    \raisebox{4.5mm}{\mbox{}}\raisebox{-2.5mm}{\mbox{}}
    \\
    \hline
    \vspace{-2mm}
    \\
    {
      \ \ \ \ \ 
      \psfrag{x}{$\pi$}\psfrag{4x}{$\frac{4\pi}{3}$}
      \psfrag{5x}{$\frac{5\pi}{3}$}\psfrag{6x}{$2\pi$}
      \psfrag{7x}{$\frac{7\pi}{3}$}\psfrag{8x}{$\frac{8\pi}{3}$}
      \psfrag{G0}{$G_0$}\psfrag{G1}{$G_1$}\psfrag{G2}{$G_2$}
      \psfrag{G3}{$G_3$}\psfrag{G4}{$G_4$}\psfrag{G5}{$G_5$}
      \pic{.7}{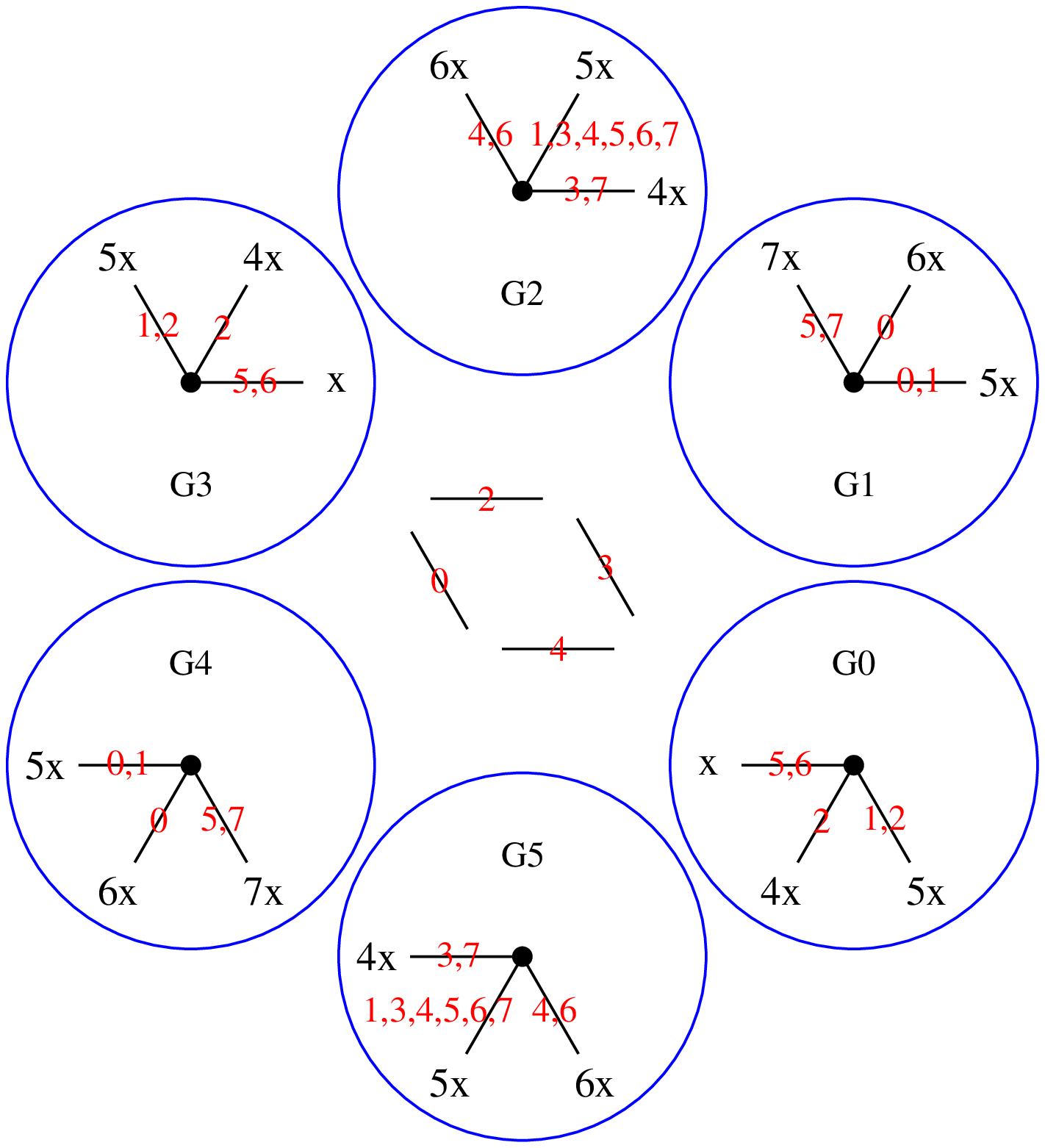}
    }
    \\
    \hline
  \end{tabular}
\end{table}

%%% Local Variables: 
%%% mode: latex
%%% TeX-master: "puzzle2eq.tex"
%%% End: 

%% file: sep02.tex
\begin{table}
  \begin{tabular}{|l|}
    \hline
    \ {\sc Table}~\ref{ssec:mutproof}(b). \ Consecutive spikes for the
    gash class $[0/2]$.
    \raisebox{4.5mm}{\mbox{}}\raisebox{-2.5mm}{\mbox{}}
    \\
    \hline
    \vspace{-3mm}
    \\
    \pic{.6}{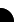}
    \pic{.6}{spike3636-6-14-0.eps} \hspace{-3mm}
    \pic{.6}{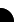} \hspace{-3mm}
    \pic{.6}{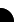}
    \pic{.6}{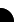}
    \pic{.6}{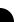}
    \pic{.6}{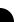} \hspace{-3mm}
    \pic{.6}{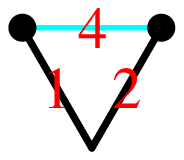}
    \pic{.6}{spike5151-2-9-1.eps} \hspace{-3mm}
    \pic{.6}{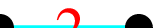}
    \\
    \hline
    \pic{.6}{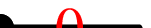}
    \pic{.6}{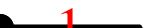}
    \pic{.6}{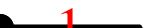}
    \pic{.6}{spike5100-0-3-1.eps}
    \pic{.6}{spike5110-8-0-1.eps}
    \pic{.6}{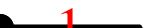}
    \pic{.6}{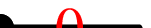}
    \pic{.6}{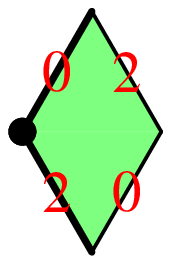}
    \\
    \hline
    \vspace{-3mm}
    \\
    \pic{.6}{spike0010-0-0-0.eps}
    \pic{.6}{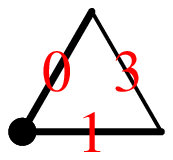} \hspace{-3mm}
    \pic{.6}{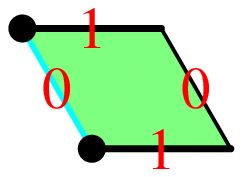} \hspace{-3mm}
    \pic{.6}{spike1000-0-0-0.eps}
    \pic{.6}{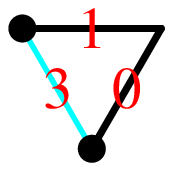}
    \pic{.6}{spike1025-2-0-5.eps}
    \pic{.6}{spike1027-4-0-7.eps} \hspace{-3mm}
    \pic{.6}{spike2510-2-0-5.eps}
    \pic{.6}{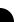} \hspace{-3mm}
    \pic{.6}{spike2727-2-15-7.eps}
    \\
    \hline
    \pic{.6}{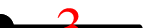}
    \pic{.6}{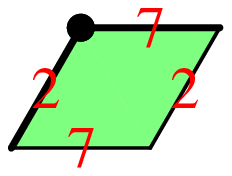}
    \pic{.6}{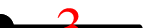}
    \pic{.6}{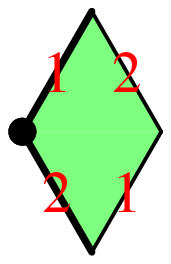}
    \pic{.6}{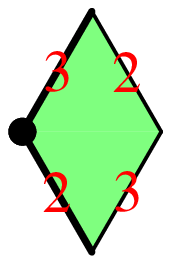}
    \pic{.6}{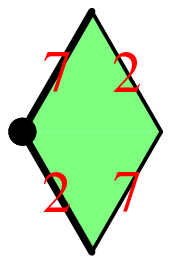}
    \pic{.6}{spike0013-0-3-1.eps}
    \pic{.6}{spike0015-0-5-2.eps}
    \pic{.6}{spike0017-0-7-4.eps}
    \\
    \pic{.6}{spike0011-0-1-8.eps}
    \pic{.6}{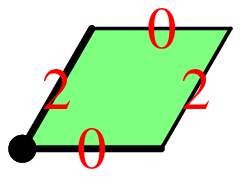}
    \pic{.6}{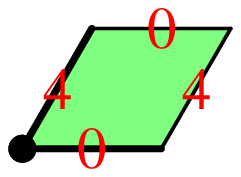}
    \pic{.6}{spike0016-0-6-14.eps}
    \pic{.6}{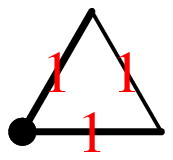}
    \pic{.6}{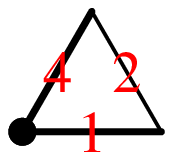}
    \pic{.6}{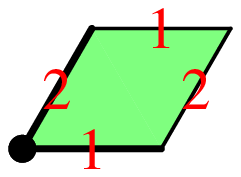}
    \\
    \pic{.6}{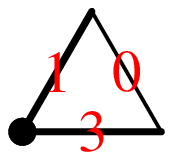}
    \pic{.6}{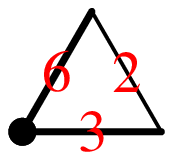}
    \pic{.6}{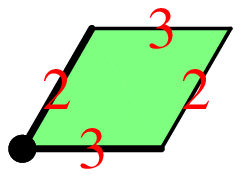}
    \pic{.6}{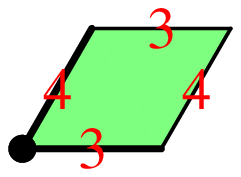}
    \pic{.6}{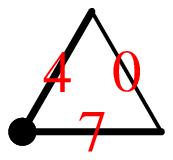}
    \pic{.6}{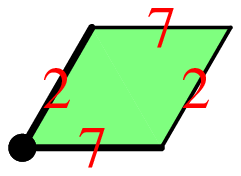}
    \\
    \pic{.6}{spike1021-8-0-1.eps}
    \pic{.6}{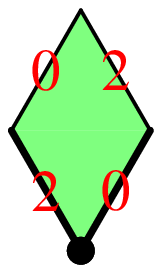}
    \pic{.6}{spike1024-12-0-4.eps}
    \pic{.6}{spike1026-14-0-6.eps}
    \pic{.6}{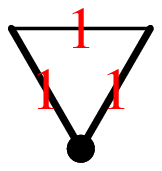}
    \pic{.6}{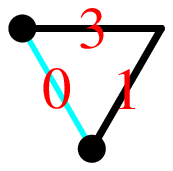}
    \pic{.6}{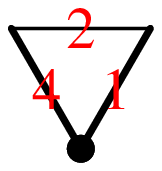}\hspace{-1mm}
    \pic{.6}{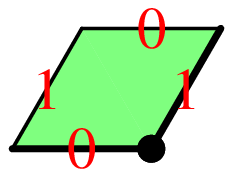}
    \pic{.6}{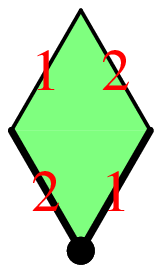}
    \pic{.6}{spike1321-0-3-1.eps}
    \\
    \pic{.6}{spike1326-2-3-6.eps}
    \raisebox{-6.5mm}{\pic{.6}{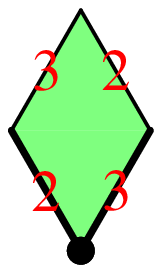}}
    \raisebox{-6.5mm}{\pic{.6}{spike1324-13-3-4.eps}}
    \pic{.6}{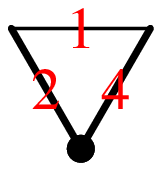}
    \pic{.6}{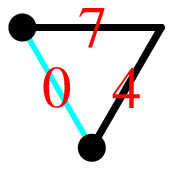}
    \pic{.6}{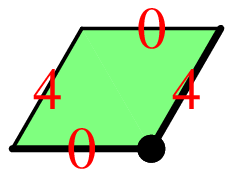}
    \pic{.6}{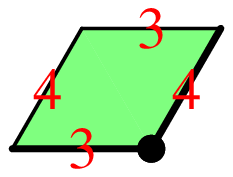}
    \pic{.6}{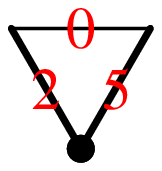}
    \pic{.6}{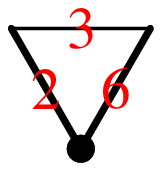}
    \vspace{-6mm}
    \\
    \pic{.6}{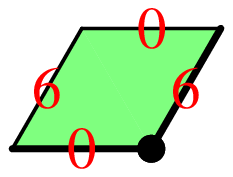} \ \
    \pic{.6}{spike1724-0-7-4.eps} \ \ \
    \pic{.6}{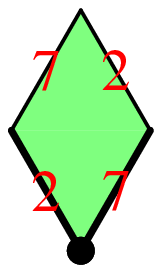}
    \pic{.6}{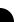}
    \pic{.6}{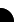}
    \pic{.6}{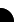}
    \pic{.6}{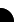}
    \pic{.6}{spike2613-2-3-6.eps}
    \pic{.6}{spike2640-14-0-6.eps}
    \\
    \hline
  \end{tabular}
\end{table}

%%% Local Variables: 
%%% mode: latex
%%% TeX-master: "puzzle2eq.tex"
%%% End: 

%% file: fig04.tex
\begin{table}
  \begin{tabular}{|l|}
    \hline
    \ {\sc Figure}~\ref{ssec:mutproof}(c). \ Spike groups and
    adjustment angles for the gash class $[0/4]$.
    \raisebox{4.5mm}{\mbox{}}\raisebox{-2.5mm}{\mbox{}}
    \\
    \hline
    \vspace{-2mm}
    \\
    {
      \ \ \ \ \ 
      \psfrag{x}{$\pi$}\psfrag{4x}{$\frac{4\pi}{3}$}
      \psfrag{5x}{$\frac{5\pi}{3}$}\psfrag{6x}{$2\pi$}
      \psfrag{7x}{$\frac{7\pi}{3}$}\psfrag{8x}{$\frac{8\pi}{3}$}
      \psfrag{G0}{$G_0$}\psfrag{G1}{$G_1$}\psfrag{G2}{$G_2$}
      \psfrag{G3}{$G_3$}\psfrag{G4}{$G_4$}\psfrag{G5}{$G_5$}
      \pic{.7}{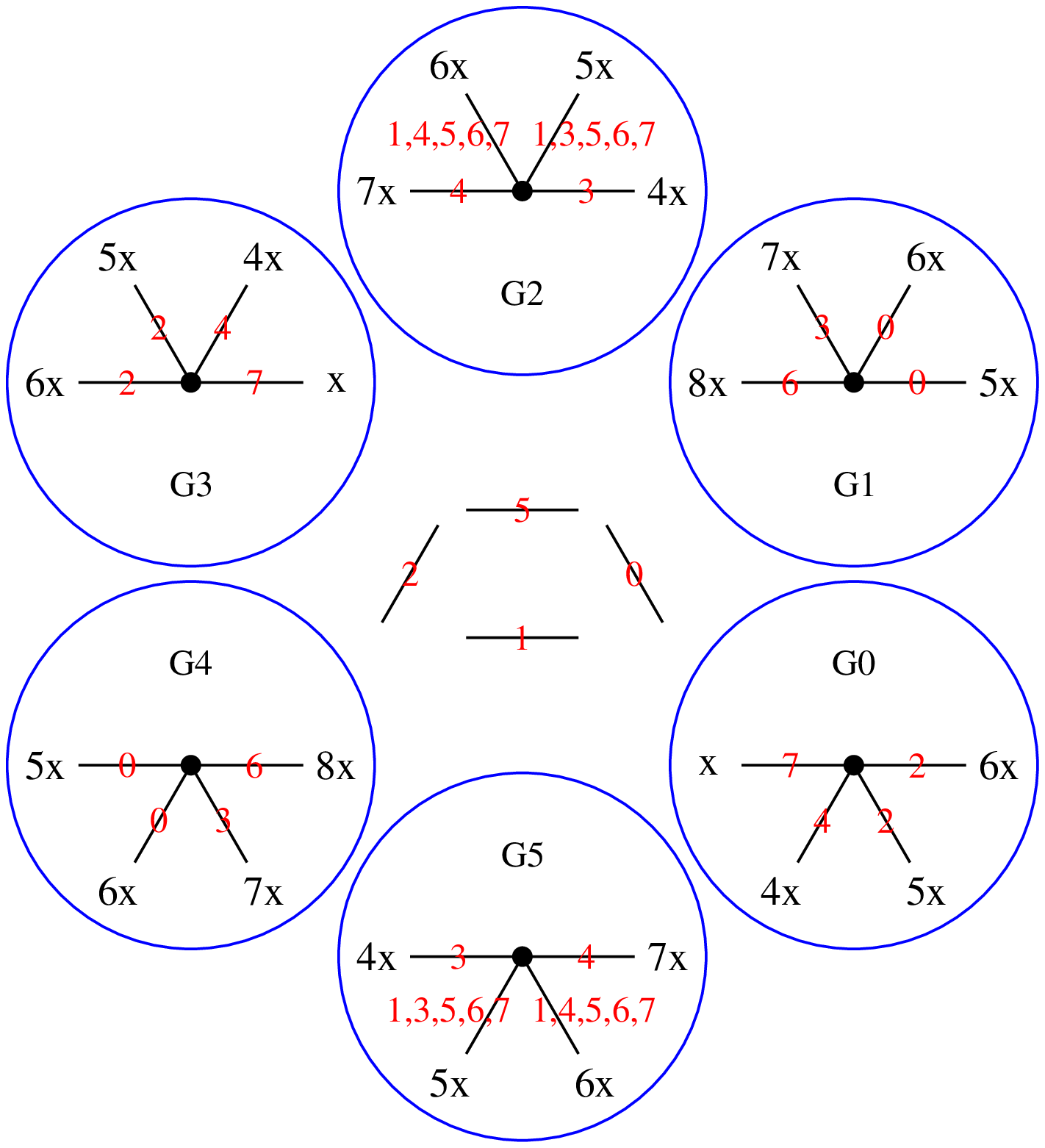}
    }
    \\
    \hline
  \end{tabular}
\end{table}

%%% Local Variables: 
%%% mode: latex
%%% TeX-master: "puzzle2eq.tex"
%%% End: 

%% file: sep04.tex
\begin{table}
  \begin{tabular}{|l|}
    \hline
    \ {\sc Table}~\ref{ssec:mutproof}(c). \ Consecutive spikes for the
    gash class $[0/4]$.
    \raisebox{4.5mm}{\mbox{}}\raisebox{-2.5mm}{\mbox{}}
    \\
    \hline
    \vspace{-3mm}
    \\
    \pic{.6}{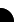}
    \pic{.6}{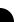}
    \pic{.6}{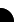}
    \pic{.6}{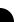}
    \pic{.6}{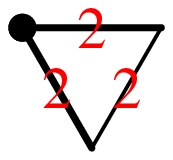}
    \pic{.6}{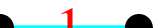}
    \pic{.6}{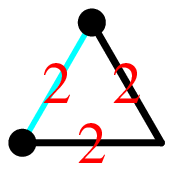}
    \pic{.6}{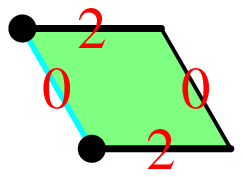}
    \\
    \hline
    \pic{.6}{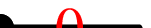}
    \pic{.6}{spike5200-0-5-2.eps}
    \pic{.6}{spike5210-10-0-2.eps}
    \pic{.6}{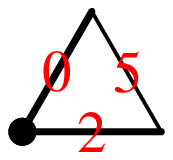}
    \pic{.6}{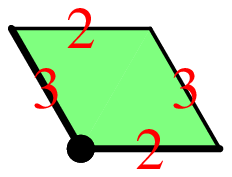}
    \\
    \hline
    \vspace{-3mm}
    \\
    \pic{.6}{spike0010-0-0-0.eps}
    \pic{.6}{spike0000-0-2-10.eps}
    \pic{.6}{spike1000-0-0-0.eps}
    \pic{.6}{spike1023-1-0-3.eps}
    \pic{.6}{spike2310-1-0-3.eps}
    \pic{.6}{spike2336-6-2-3.eps}
    \pic{.6}{spike3623-6-2-3.eps}
    \pic{.6}{spike3636-6-14-0.eps}
    \\
    \hline
    \pic{.6}{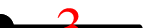}
    \pic{.6}{spike5203-3-6-2.eps}
    \pic{.6}{spike5211-9-1-2.eps}
    \pic{.6}{spike5213-11-3-2.eps}
    \pic{.6}{spike5217-15-7-2.eps}
    \pic{.6}{spike0013-0-3-1.eps}
    \pic{.6}{spike0015-0-5-2.eps}
    \pic{.6}{spike0017-0-7-4.eps}
    \pic{.6}{spike0011-0-1-8.eps}
    \\
    \pic{.6}{spike0014-0-4-12.eps}
    \pic{.6}{spike0016-0-6-14.eps}
    \pic{.6}{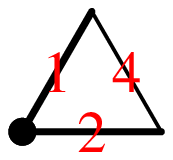}
    \pic{.6}{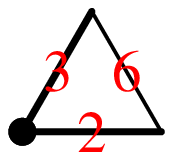}
    \pic{.6}{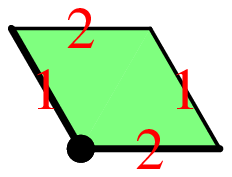}
    \pic{.6}{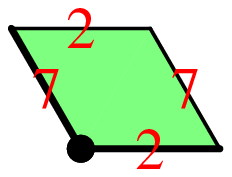}
    \pic{.6}{spike0311-3-1-0.eps}
    \pic{.6}{spike0316-3-6-2.eps}
    \\
    \pic{.6}{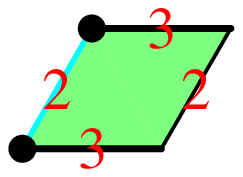}
    \pic{.6}{spike0314-3-4-13.eps}
    \pic{.6}{spike1025-2-0-5.eps}
    \pic{.6}{spike1027-4-0-7.eps}
    \pic{.6}{spike1021-8-0-1.eps}
    \pic{.6}{spike1022-10-0-2.eps}
    \pic{.6}{spike1024-12-0-4.eps}
    \pic{.6}{spike1026-14-0-6.eps}
    \pic{.6}{spike1121-1-1-1.eps}
    \\
    \pic{.6}{spike1103-3-1-0.eps}
    \pic{.6}{spike1124-2-1-4.eps}
    \pic{.6}{spike1130-0-1-8.eps}
    \pic{.6}{spike1122-9-1-2.eps}
    \pic{.6}{spike1321-0-3-1.eps}
    \pic{.6}{spike1326-2-3-6.eps}
    \pic{.6}{spike1322-11-3-2.eps}
    \pic{.6}{spike1324-13-3-4.eps}
    \pic{.6}{spike1522-0-5-2.eps}
    \\
    \pic{.6}{spike1622-3-6-2.eps}
    \pic{.6}{spike1630-0-6-14.eps}
    \pic{.6}{spike1724-0-7-4.eps}
    \raisebox{-6.5mm}{\pic{.6}{spike1722-15-7-2.eps}}
    \pic{.6}{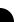}
    \pic{.6}{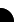}
    \pic{.6}{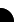}
    \raisebox{-6.5mm}{\pic{.6}{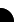}}
    \pic{.6}{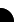}
    \vspace{-6mm}
    \\
    \pic{.6}{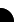}
    \pic{.6}{spike2334-4-13-3.eps}
    \pic{.6}{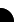}
    \pic{.6}{spike2430-0-7-4.eps}
    \pic{.6}{spike2440-12-0-4.eps}
    \pic{.6}{spike2443-13-3-4.eps}
    \pic{.6}{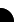}
    \pic{.6}{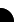}
    \pic{.6}{spike2640-14-0-6.eps}
    \\
    \pic{.6}{spike2734-4-0-7.eps}
    \pic{.6}{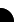}
    \pic{.6}{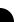}
    \pic{.6}{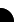}
    \pic{.6}{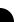}
    \pic{.6}{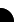}
    \\
    \hline
  \end{tabular}
\end{table}

%%% Local Variables: 
%%% mode: latex
%%% TeX-master: "puzzle2eq.tex"
%%% End: 

%% file: biject.tex
\subsection{Bijections of puzzles}\label{ssec:biject}

We finish this section by explaining how the mutation algorithm can be
used to give new constructions of certain bijections of puzzles
defined in the papers \cite{knutson.tao.ea:honeycomb,
  knutson.tao:puzzles, buch.kresch.ea:puzzle}.  These constructions
are not required for our proof of Theorem~\ref{thm:puzzle}.  We start
by generalizing the bijections from \cite{knutson.tao:puzzles,
  buch.kresch.ea:puzzle} which were applied to prove special cases of
Theorem~\ref{thm:puzzle}.

Let $\Gright$ be the union of equivalence classes of gashes defined by
\[
\Gright \ = \ 
\left[ \raisebox{-4mm}{\pic{.7}{gash110}} \right] \cup
\left[ \raisebox{-4mm}{\pic{.7}{gash120}} \right] \cup
\left[ \raisebox{-4mm}{\pic{.7}{gash140}} \right] \cup
\left[ \raisebox{-4mm}{\pic{.7}{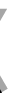}} \right] \cup
\left[ \raisebox{-4mm}{\pic{.7}{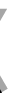}} \right] \cup
\left[ \raisebox{-4mm}{\pic{.7}{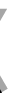}} \right] \,.
\]
Let $\Rright$ be the set of all resolutions of flawed puzzles for
which the right gash belongs to $\Gright$, and let $\Pright$ be the
set of all flawed puzzles for which at least one resolution belongs to
$\Rright$.  We also let $\Rleft$ and $\Pleft$ denote the sets obtained
by rotating the objects in $\Rright$ and $\Pright$ by 180 degrees.
Given any set of flawed puzzles $S$, we write $S_\gash$, $S_\scab$,
and $S_\temp$ for the subsets of puzzles in $S$ whose flaws have
the indicated types.  Notice that the gash pair of any puzzle in
$\Pright_{\!\gash}$ is located on one of the left border segments,
while the gash pair of a puzzle in $\Pleft_{\!\gash}$ is located on
one of the right border segments.

\begin{lemma}\label{lem:biject}
  We have $\Pright_{\!\temp} = \Pleft_{\temp} = \Pright \cap \Pleft$.
  Furthermore, any puzzle in $\Pright \cap \Pleft$ has exactly one
  resolution in $\Rright$ and exactly one resolution in $\Rleft$.
\end{lemma}
\begin{proof}
  The (right-side-up) temporary puzzle pieces that occur in
  $\Pright_{\!\temp}$, and the resolutions of these pieces that
  provide elements of $\Rright$ and $\Rleft$, are listed in
  Table~\ref{ssec:biject}.
\end{proof}

\input{bijtab}
\bigskip

The involution $\Phi$ defined in Section~\ref{ssec:mutations}
restricts to a bijection from $\Rright$ to $\Rleft$.  We can therefore
define a bijection $\Psi : \Pright \to \Pleft$ as follows.  Given $P
\in \Pright$, let $\wtil P$ be the unique resolution of $P$ that
belongs to $\Rright$, and let $\Psi(P)$ be the unique flawed puzzle
that has $\Phi(\wtil P)$ as a resolution.  Notice that if $\Psi(P) \in
\Pright \cap \Pleft$, then we may apply $\Psi$ an additional time.
Let $\Psi^\infty(P)$ denote the result of applying $\Psi$ to $P$ until
we obtain a flawed puzzle in the set $\Pleft \ssm \Pright =
\Pleft_{\!\gash} \cup \Pleft_{\!\scab}$.  The restriction of
$\Psi^\infty$ to $\Pright \ssm \Pleft$ is a bijection
\[
\Psi^\infty : \Pright_{\!\gash} \cup \Pright_{\!\scab} 
\xra{\ \simeq \ }
\Pleft_{\!\gash} \cup \Pleft_{\!\scab} \,.
\]
For example, $\Psi^\infty$ maps the top-left puzzle in
Figure~\ref{ssec:mutations} to the top-right puzzle, and it maps the
bottom-left puzzle to the middle-right puzzle.  Related bijections can
be obtained by conjugating $\Psi^\infty$ by rotations and/or
dualization of flawed puzzles.  This corresponds to rotating and/or
dualizing the gashes in $\Gright$.

The bijections of puzzles from \cite{knutson.tao:puzzles,
  buch.kresch.ea:puzzle} related to multiplication with divisors are
special cases of $\Psi^\infty$ and its conjugates.  Notice that our
definition of $\Psi^\infty$ involves modifying some areas of a puzzle
multiple times.  In contrast the constructions used in
\cite{knutson.tao:puzzles, buch.kresch.ea:puzzle} directly describe
the end results of the respective bijections.  By factoring the
bijection $\Psi^\infty$ into a series of mutations, we have obtained a
simpler and more conceptual description.

\begin{remark}
  The classical Littlewood-Richardson rule expresses any
  Littlewood-Richardson coefficient $c^\nu_{\la,\mu}$ as the number of
  {\em LR tableaux\/} of shape $\nu/\la$ and weight $\mu$.  The
  precise definitions can be found in e.g.\ \cite{fulton:young}.  By
  composing bijections of Fulton \cite{buch:saturation} and of
  Knutson, Tao, and Woodward \cite{knutson.tao.ea:honeycomb}, one may
  obtain a bijection between these LR tableaux and the set of puzzles
  counted by the cohomological puzzle rule for Grassmannians.  A more
  general bijection between equivariant LR tableaux and equivariant
  puzzles for Grassmannians has been defined by Kreiman
  \cite{kreiman:equivariant}.  Given a LR tableau $T$ of shape
  $\nu/\la$ and a Young diagram $\la' \subset \la$ with one box less
  than $\la$, the jeu de taquin algorithm can be used to produce a LR
  tableau $T'$ of some shape $\nu'/\la'$, where $\nu' \subset \nu$ has
  one box less than $\nu$.  The bijection $\Psi^\infty$ is compatible
  with the jeu de taquin algorithm in the sense that the puzzle
  corresponding to $T'$ may be obtained from the puzzle corresponding
  to $T$ by applying one of the conjugates of $\Psi^\infty$.  However,
  it is not possible to extend the bijection between LR tableaux and
  puzzles to a bijection between tableaux with empty boxes and flawed
  puzzles in a way such that individual jeu de taquin slides
  correspond to individual mutations.  For example, if the box of
  $\la/\la'$ is an outer corner of $\nu$, then the jeu de taquin
  algorithm involves zero slides, whereas an arbitrary number of
  mutations may be required to transform the corresponding puzzles.
  Similarly one can construct examples where two mutations correspond
  to an arbitrary number of jeu de taquin slides.  Notice also that
  not all conjugates of $\Psi^\infty$ correspond to the jeu de taquin
  algorithm.
\end{remark}

\subsection{Breathing gentle loops}

We finally address a construction of Knutson, Tao, and Woodward that
was used in \cite{knutson.tao.ea:honeycomb} to characterize
Littlewood-Richardson coefficients equal to one and to prove a related
conjecture of Fulton.  Recall from \cite{knutson.tao.ea:honeycomb}
that any Littlewood-Richardson coefficient $c^\nu_{\la,\mu}$ counts
puzzles made from the pieces \raisebox{-10pt}{\pic{.6}{u000}},
\raisebox{-10pt}{\pic{.6}{u111}}, and
\raisebox{-10pt}{\pic{.6}{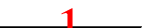}}.  A {\em gentle loop\/} in such
a puzzle is defined to be an oriented cycle of puzzle edges, with
turns of $\pm 60^\circ$, such that each edge in the cycle separates
two puzzle pieces of different types.  In addition, each edge must be
directed so that it has either a 0-triangle on its left side or a
1-triangle on its right side.  It is proved in
\cite[Lemma~6]{knutson.tao.ea:honeycomb} that, if $\gamma$ is any
gentle loop of minimal length in a Grassmannian puzzle, then a new
valid puzzle can be obtained by replacing all puzzle pieces in the
radius-1 neighborhood of $\gamma$ with different pieces.  This is
called {\em breathing\/} the gentle loop, and it demonstrates that the
corresponding Littlewood-Richardson coefficient must be at least 2.
The breathing construction in \cite{knutson.tao.ea:honeycomb} is
defined by specifying how to modify each local region in the radius-1
neighborhood of $\gamma$.  We will sketch how a minimal gentle loop
can also be breathed by applying a sequence of mutations.  We thank
the referee for providing this application.

Given a minimal gentle loop in a puzzle, consider a {\em normal
  line\/} consisting of two puzzle edges of equal slope that cuts
across the loop (see \cite[\S 4.2]{knutson.tao.ea:honeycomb}).  On
this normal line we place two gash pairs infinitesimally close to each
other, so that the outer labels agree with the original labels of the
normal line.  The gentle loop can then be breathed by propagating one
of the gash pairs around the loop.  Whenever a temporary puzzle piece
is created, this piece must be resolved in the direction of the gentle
loop.  Eventually the moving gash pair will reach the other side of
the normal line, where it cancels the stationary gash pair.  It is
important to perform the propagations in the direction of the gentle
loop, as otherwise the process will run astray.  Notice also that,
while resolutions of temporary puzzle pieces in the construction of
$\Psi^\infty$ are chosen to keep the propagations moving in a constant
direction, the breathing construction chooses resolutions that steer
the propagations around the loop.

\begin{example}
  The two shortest gentle loops have length 6 and are interchanged by
  breathing.  We list the initial and terminal double-gashed puzzles
  as well as all intermediate puzzles that contain both a temporary
  puzzle piece and the stationary gash pair.  Notice that the gentle
  loop changes orientation during the process, and that the normal
  line can be chosen in several ways.\smallskip

  \noindent
  \pic{.53}{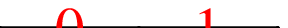}
  \pic{.53}{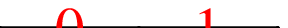}
  \pic{.53}{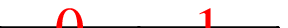}
  \pic{.53}{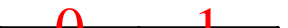}

  \noindent
  \pic{.53}{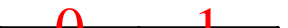}
  \pic{.53}{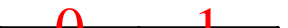}
  \pic{.53}{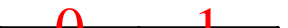}
  \pic{.53}{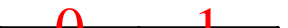}
\end{example}

%%% Local Variables: 
%%% mode: latex
%%% TeX-master: "puzzle2eq.tex"
%%% End: 

%% file: bijtab.tex
\def\bijtabsep{\hspace{6.5mm}}

\noindent
\begin{center}
  \begin{tabular}{|l|}
    \hline
    \ {\sc Table}~\ref{ssec:biject}. \ Temporary puzzle pieces
    encountered in $\Pright \cap \Pleft$.
    \raisebox{5.0mm}{\mbox{}}\raisebox{-2.5mm}{\mbox{}}
    \\
    \hline
    \vspace{-1mm}
    \\
    \pic{.65}{u333r2} 
    \pic{.65}{u333} 
    \pic{.65}{u333r1} 
    \bijtabsep
    \pic{.65}{u555r2} 
    \pic{.65}{u555} 
    \pic{.65}{u555r1} 
    \bijtabsep
    \pic{.65}{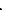} 
    \pic{.65}{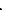} 
    \pic{.65}{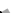} 
    \\
    \\
    \pic{.65}{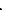}
    \pic{.65}{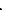}
    \pic{.65}{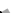}
    \bijtabsep
    \pic{.65}{u645r2} 
    \pic{.65}{u645} 
    \pic{.65}{u645r1} 
    \bijtabsep
    \pic{.65}{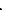} 
    \pic{.65}{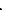} 
    \pic{.65}{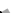} 
    \\
    \\
    \pic{.65}{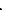}
    \pic{.65}{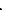}
    \pic{.65}{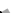}
    \bijtabsep
    \pic{.65}{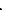}
    \pic{.65}{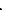}
    \pic{.65}{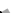}
    \bijtabsep
    \pic{.65}{u753r2} 
    \pic{.65}{u753} 
    \pic{.65}{u753r1} 
    \vspace{-1mm}
    \\
    \\
    \hline
  \end{tabular}
\end{center}

%%% Local Variables: 
%%% mode: latex
%%% TeX-master: "puzzle2eq.tex"
%%% End: 

%% file: aura.tex
\section{Auras of puzzles and the proof of the puzzle
  formula}\label{sec:aura}

\subsection{Aura}\label{ssec:aura}

In this section we assign an {\em aura\/} to certain objects related
to puzzles and use this concept together with the mutation algorithm
to prove Theorem~\ref{thm:puzzle}.  An aura is a linear form in the
ring $R = \C[\delta_0,\delta_1,\delta_2]$ from
Section~\ref{sec:recurse}.  We will represent auras graphically as a
collection of unit vectors labeled with linear forms from
$\Z[\delta_0,\delta_1,\delta_2]$.  The aura is then the sum of the
unit vectors multiplied to their labels.  For example, we have
\[
\psfrag{d0}{$\delta_0$}
\psfrag{d1}{$\delta_1$}
\psfrag{d2}{$\delta_2$}
\delta_0 \zeta + \delta_1 \zeta^5 + \delta_2 \zeta^9 \ = \ 
\raisebox{-8mm}{\pic{1}{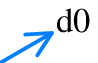}}
\]
where $\zeta = \exp(\pi i / 6) \in \C$.

Define a {\em semi-labeled edge\/} to be a puzzle edge that has a
label only on one side.  We will use the textual notation $a/$, $/a$,
$a\backslash$, $\backslash a$,
$\ds\frac{\,\raisebox{-.5mm}{$a$}\,}{}$, and
$\ds\frac{}{\,\raisebox{.8mm}{$a$}\,}$ for such edges.  The aura
$\cA(e)$ of a semi-labeled edge $e$ is defined as follows.  If the
label $a$ of $e$ is simple, then we set $\cA(e) = \delta_a v$, where
$v \in \C$ is a unit vector perpendicular to $e$ that points towards
the side of the label.  Otherwise $\cA(e)$ is determined by the rule
that, whenever the sides of a valid puzzle piece are changed to
semi-labeled edges by moving their labels slightly inside the puzzle
piece, the sum of the auras of the sides is zero.  For example, using
the puzzle piece \raisebox{-4.5mm}{\pic{.6}{u310}} we obtain
\[
\psfrag{d0}{$\delta_0$}
\psfrag{d1}{$\delta_1$}
\psfrag{d2}{$\delta_2$}
\cA(\frac{\,3\,}{}) =
- \cA(/1) - \cA(0\backslash) =
\delta_1 \zeta^5 + \delta_0 \zeta =
\raisebox{-2mm}{\pic{1}{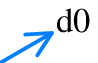}} \,.
\]
The auras of all semi-labeled edges can be obtained by rotating the
following identities.
{
\psfrag{d0}{$\delta_0$}
\psfrag{d1}{$\delta_1$}
\psfrag{d2}{$\delta_2$}
\begin{align*}
  \cA(\frac{\,0\,}{}) &= \raisebox{-4mm}{\pic{1}{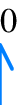}} &
  \cA(\frac{\,1\,}{}) &= \raisebox{-4mm}{\pic{1}{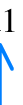}} &
  \cA(\frac{\,2\,}{}) &= \raisebox{-4mm}{\pic{1}{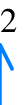}} \\
  \cA(\frac{\,3\,}{}) &= \raisebox{-2mm}{\pic{1}{aura3u}} &
  \cA(\frac{\,4\,}{}) &= \raisebox{-2mm}{\pic{1}{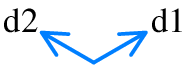}} &
  \cA(\frac{\,5\,}{}) &= \raisebox{-2mm}{\pic{1}{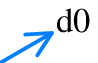}} \\
  \cA(\frac{\,6\,}{}) &= \raisebox{-7mm}{\pic{1}{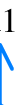}} &
  \cA(\frac{\,7\,}{}) &= \raisebox{-7mm}{\pic{1}{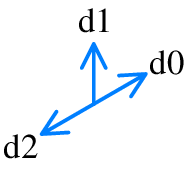}} 
\end{align*}
}

A gash can be regarded as a union of two semi-labeled edges.  We
define the aura of a gash to be the sum of the auras of the two
semi-labeled edges.  For example, we have
\[
\psfrag{d0}{$\delta_0$}
\psfrag{d1}{$\delta_1$}
\psfrag{d2}{$\delta_2$}
\cA(\frac{\,0\,}{\,4\,}) \ = \ 
\cA(\frac{\,0\,}{}) + 
\cA(\frac{}{\,4\,}) \ = \ 
\delta_0 \zeta^3 + \delta_1 \zeta^7 + \delta_2 \zeta^{11} \ = \ 
\raisebox{-7mm}{\pic{1}{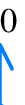}} \,.
\]
The aura of a directed gash is the aura of the underlying undirected
gash.  The following are additional examples of auras of gashes.
\[
{
\psfrag{d0}{$\!\!\!\!\!\delta_0\!-\!\delta_1$}%
\cA(\frac{\,0\,}{\,1\,}) = \ \raisebox{-4mm}{\pic{1}{aura0u}}
}
\ \ \ \ \ \ \ \ 
{
\psfrag{d0}{$\!\!\!\!\!\delta_0\!-\!\delta_2$}%
\cA(\frac{\,0\,}{\,2\,}) = \ \raisebox{-4mm}{\pic{1}{aura0u}}
}
\ \ \ \ \ \ \ \ 
{
\psfrag{d0}{$\!\!\!\!\!\delta_1\!-\!\delta_2$}%
\cA(\frac{\,1\,}{\,2\,}) = \ \raisebox{-4mm}{\pic{1}{aura0u}}
}
\ \ \ \ \ \ \ \ 
{
\psfrag{d0}{$\delta_0$}%
\psfrag{d1}{$\delta_1$}%
\psfrag{d2}{$\delta_2$}%
\cA(\frac{\,3\,}{\,2\,}) = \!\!\!\raisebox{-7mm}{\pic{1}{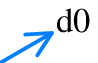}}
}
\]

\begin{lemma}\label{lem:class_aura}
  Any two gashes in the same gash class have the same aura.
\end{lemma}
\begin{proof}
  Let $g$ and $h$ be gashes that are immediately reachable from each
  other.  We must show that $\cA(g) = \cA(h)$.  After rotating and
  possibly interchanging the gashes, we may assume that $g = a/b$ and
  $h = x\backslash y$.  Furthermore, the labels of the gashes appear
  on puzzle pieces of the form:
\[
{
\psfrag{a}{$x$}
\psfrag{e}{$a$}
\psfrag{c}{$c$}
\pic{.7}{uaec2}
}
\raisebox{4mm}{\ \ \ \ \text{and}\ \ \ \ }
{
\psfrag{a}{$y$}
\psfrag{e}{$b$}
\psfrag{c}{$c$}
\pic{.7}{uaec2}
}
\raisebox{4mm}{\ .}
\]
By definition of the aura of semi-labeled edges we therefore obtain
\[
\cA(a/b) 
= 
\cA(a/) + \cA(/b) 
= 
\cA(\frac{\,c\,}{}) + \cA(x\backslash) + 
\cA(\frac{}{\,c\,}) + \cA(\backslash y) 
= 
\cA(x\backslash y) \,,
\]
as required.
\end{proof}

Let $P$ be a flawed puzzle and let $\wtil P$ be a resolution of $P$.
We define $\cA(\wtil P)$ to be the aura of the right gash of $\wtil
P$.  If the flaw in $P$ is a gash pair or a marked scab, so that
$\wtil P$ is the only resolution of $P$, then we also write $\cA(P) =
\cA(\wtil P)$.  Recall that, if $S$ is any set of flawed puzzles, then
we write $S_{\gash}$, $S_{\scab}$, and $S_{\temp}$ for the subsets of
puzzles with flaws of the indicated types.  Our main application of
the mutation algorithm is the following identity, which is proved in
the generality of hexagonal puzzles with equivariant puzzle pieces in
arbitrary orientations.  The two sums in this identity will later be
related to the two sides of the recursive identity
(\ref{eqn:recurse}).

\begin{prop}\label{prop:sum_flawed}
  Let $S$ be any finite set of flawed puzzles that is closed under
  mutations.  Then we have
  \[ 
  \sum_{P \in S_{\scab}} \cA(P) \,+ \sum_{P \in S_{\gash}} \cA(P) \ =
  \ 0 \,.
  \]
\end{prop}
\begin{proof}
  Let $\wtil S$ be the set of all resolutions of the flawed puzzles in
  $S$.  Since $S$ is closed under mutations, it follows that the
  involution $\Phi$ defined in Section~\ref{ssec:mutations} restricts
  to an involution of $\wtil S$.  Since Lemma~\ref{lem:class_aura}
  implies that $\cA(\wtil P) + \cA(\Phi(\wtil P)) = 0$ for any $\wtil
  P \in \wtil S$, we deduce that
  \[
  \sum_{\wtil P \in \wtil S} \cA(\wtil P) = 0 \,.
  \]

  It suffices to show that, if $P$ is any flawed puzzle containing a
  temporary puzzle piece, then the sum of the auras of the three
  resolutions of $P$ is equal to zero.  Assume that $P$ contains the
  temporary piece $t$ displayed in Definition~\ref{def:temporary}, and
  let the labels $x,x',x'',y,y',y'',z,z',z''$ and the puzzle pieces
  $r_1, r_2, r_3, t', t''$ be as in this definition.  Then the right
  gashes of the three resolutions of $P$ are $y/y'$, $x'\backslash x$,
  and $\ds\frac{\,z'\,}{z}$.  Thanks to the puzzle pieces $r_3$,
  $r_1$, and $r_2$ we have
  \[
  \begin{split}
    \cA(y/y') &=
    \cA(x''\backslash) + \cA(/y') + \cA(\frac{\,z'\,}{}) \ , \\
    \cA(x'\backslash x) &= 
    \cA(x'\backslash) + \cA(/y') + \cA(\frac{\,z''\,}{}) \ \text{,
      and} \\
    \cA(\frac{\,z'\,}{z}) &=
    \cA(x'\backslash) + \cA(/y'') + \cA(\frac{\,z'\,}{}) \,.
  \end{split}
  \]
  The last two puzzle pieces $t'$ and $t''$ therefore imply that
  \[
  \cA(y/y') + \cA(x'\backslash x) + \cA(\frac{\,z'\,}{z}) = 0 \,.
  \]
  This completes the proof.
\end{proof}

\subsection{The constants $C^w_{w,w}$}\label{ssec:Cwww}

We first apply the notion of aura to prove that the equivariant puzzle
rule is compatible with restrictions of Schubert classes to torus
fixed points.  Let $X = \Fl(a,b;n)$ be a two-step flag variety.

\begin{lemma}\label{lem:012side}
  Let $P$ be any triangle made from puzzle pieces (in any orientation)
  with matching side labels, and let $u$, $v$, and $w$ be strings of
  labels such that $\partial P = \border^{u,v}_w$.  If $u$ and $v$ are
  012-strings for $X$, then so is $w$.  In particular, $w$ consists of
  simple labels.
\end{lemma}
\begin{proof}
  Consider all pairs $(q,e)$ where $q$ is a puzzle piece in $P$ and
  $e$ is a side of $q$.  For each such pair we regard $e$ as a
  semi-labeled edge, where the label is slightly inside the puzzle
  piece $q$.  Now consider the sum
  \[
  \phi = \sum_{(q,e)} \cA(e)
  \]
  over all such pairs.  Since the sum over the sides of each puzzle
  piece $q$ is zero, we have $\phi = 0$.  On the other hand, since
  each interior edge of $P$ appears twice in the sum with its label on
  opposite sides, it follows that the sum of the auras of all boundary
  edges of $P$ is equal to zero.  Set $\gamma = a \delta_0 + (b-a)
  \delta_1 + (n-b) \delta_2$.  The assumption that $u$ and $v$ are
  012-strings for $X$ implies that the sum of the auras of the left
  border edges is equal to $\gamma\,\zeta^{11}$, and the sum of the
  auras of the right border edges is equal to $\gamma\,\zeta^7$.  We
  deduce that
  \begin{equation}\label{eqn:wsimple}
    \sum_{i=1}^n \cA(\frac{\,w_i\,}{}) = \gamma\,\zeta^3 \,.
  \end{equation}
  Since the coefficient of $\delta_0$ in this expression is a multiple
  of the vertical vector $\zeta^3$, an inspection of the auras of
  horizontal semi-labeled edges listed in Section~\ref{ssec:aura}
  shows that $w$ does not contain any of the labels 3, 5, 6, and 7.
  Similarly, since the coefficient of $\delta_2$ is a multiple of
  $\zeta^3$, we deduce that $w$ does not contain any of the labels 4,
  5, 6, and 7.  It follows that $w$ consists of simple labels, after
  which (\ref{eqn:wsimple}) shows that $w$ is a 012-string for $X$.
\end{proof}

Let $P$ be an equivariant puzzle for $X$ and recall from
Section~\ref{sec:equiv} that we number the edges of the bottom border
segment from $1$ to $n$, starting from the left.  An edge in $P$ will
be called SW-NE if it is parallel to the left border segment, NW-SE if
it is parallel to the right border segment, and horizontal otherwise.
Given any NW-SE edge $e$ in $P$, define the {\em left projection\/} of
$e$ to be the number of the bottom edge obtained by following a line
parallel to the left border segment.  Similarly, the {\em right
  projection\/} of a SW-NE edge is the number of the bottom edge
obtained by following a line parallel to the right border segment.

Since all equivariant puzzle pieces in $P$ are vertical, we may
dissect $P$ into $\binom{n}{2}$ small vertical rhombuses together with
$n$ triangular puzzle pieces along the bottom border.  Each small
vertical rhombus $s$ is either an equivariant puzzle piece or the
union of two triangular puzzle pieces.  We will say that $s$ is in
position $(i,j)$ if $i$ is the left projection of its NW-SE edges and
$j$ is the right projection of its SW-NE edges.  In this case we
define the {\em weight\/} of $s$ to be $\weight(s) = y_j-y_i$.  This
extends the definition of the weight of an equivariant puzzle piece
given in Section~\ref{sec:equiv}.

The following result implies that the constants $C^w_{u,v}$ defined by
the equivariant puzzle rule satisfy equation (\ref{eqn:extreme})
from Theorem~\ref{thm:recurse}.

\begin{prop}\label{prop:extreme}
  Let $w$ be any 012-string for $X$.  Then there exists a unique
  equivariant puzzle $P$ for $X$ with boundary $\border^{w,w}_w$, and
  this puzzle satisfies
  \[
  \weight(P) \ = \ \prod_{i<j :\, w_i>w_j} (y_j - y_i) \ \in \Lambda \,.
  \]
\end{prop}
\begin{proof}
  Let $P$ be any equivariant puzzle for $X$ with $\partial P =
  \border^{w,w}_w$, and consider any separation of $P$ into two
  subpuzzles by any NW-SE line that goes along puzzle edges.
  \[
  \psfrag{w1}{$w_1$}
  \psfrag{w2}{$w_2$}
  \psfrag{w3}{$w_3$}
  \psfrag{w4}{$w_4$}
  \psfrag{w5}{$w_5$}
  \pic{.7}{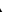}
  \]
  Notice that two of the border segments of the triangular subpuzzle
  are 012-strings for the same two-step flag variety.  It therefore
  follows from Lemma~\ref{lem:012side} that all labels on the
  separating line are simple.  We deduce that all NW-SE puzzle edges
  in $P$ have simple labels, and a symmetric argument shows that all
  SW-NE edges have simple labels.  In particular, each small vertical
  rhombus in $P$ has simple border labels.  An inspection of the
  puzzle pieces from Section~\ref{sec:equiv} shows that, if all border
  labels of a small rhombus are simple, then opposite border edges
  have the same label.  We deduce that the border labels of the small
  vertical rhombus in position $(i,j)$ are given by:
  \[
  \psfrag{a}{\raisebox{.5mm}{$\!w_j$}}
  \psfrag{c}{\raisebox{.5mm}{$\!w_i$}}
  \psfrag{z}{\raisebox{.5mm}{$\!w_j$}}
  \psfrag{x}{\raisebox{.5mm}{$\!w_i$}}
  \pic{.8}{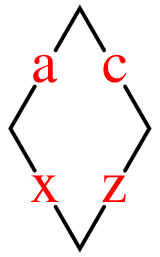}
  \]
  This shows that $P$ is uniquely determined by its boundary, and also
  provides a recipe for constructing $P$.  Finally, the expression for
  $\weight(P)$ is correct because the small vertical rhombus in
  position $(i,j)$ is an equivariant puzzle piece if and only if $w_i
  > w_j$.
\end{proof}

\subsection{Equivariant Aura}\label{ssec:equiv_aura}

Let $P$ be an equivariant puzzle for $X = \Fl(a,b;n)$ and let $e$ be
an edge in $P$.  If $e$ is a NW-SE edge, then we set $\weight(e) =
y_i$ where $i$ is the left projection of $e$.  If $e$ is a SW-NE edge,
then set $\weight(e) = y_j$ where $j$ is the right projection of $e$.
Finally, if $e$ is a horizontal edge, then we set $\weight(e) = y_i$
if $e$ is the $i$-th edge of the bottom border segment, and otherwise
we set $\weight(e) = 0$.

An {\em equivariant aura\/} is an element of the ring $R[y] =
R[y_1,\dots,y_n]$.  If $e$ is a semi-labeled edge in $P$, then we
define the equivariant aura of $e$ to be $\cA_T(e) = \weight(e)
\cA(e)$.  Given any puzzle piece $q$ in $P$ we let $\cA_T(q)$ be the
sum of the equivariant auras of the sides of $q$, where these sides
are regarded as semi-labeled edges by moving their labels slightly
inside $q$.  If $s$ is any small vertical rhombus in $P$ consisting of
two triangular puzzle pieces, then we let $\cA_T(s)$ be the sum of the
equivariant auras of these pieces.

\begin{prop}\label{prop:sum_equiv_aura}
  Let $u$, $v$, and $w$ be 012-strings for $X$, and let $P$ be an
  equivariant puzzle for $X$ with boundary $\partial P =
  \triangle^{u,v}_w$.  Then we have
  \[
  \sum_{s \in \scabs(P)} \cA_T(s) \ = \ C_u \zeta^{11} + C_v \zeta^7 + C_w \zeta^3
  \]
  where the sum is over all vertical scabs in $P$.
\end{prop}
\begin{proof}
  Consider the sum $\phi = \sum_q \cA_T(q)$ over all puzzle pieces $q$
  in $P$.  Since the equivariant aura of all inner puzzle edges
  cancel, $\phi$ is equal to the right hand side of the claimed
  identity.  On the other hand, if $s$ is any vertical rhombus in $P$
  that is not a scab, then $\cA_T(s) = 0$.  In addition we have
  $\cA_T(q) = 0$ whenever $q$ is a triangular puzzle piece on the
  bottom border of $P$.  This implies that $\phi$ is equal to the left
  hand side of the claimed identity.
\end{proof}

A flawed puzzle $P$ is called a {\em flawed puzzle for $X$\/} if $P$
is a right-side-up triangle with boundary $\border^{u,v}_w$ where $u$,
$v$, and $w$ are 012-strings for $X$, and all equivariant puzzle
pieces and marked scabs in $P$ are vertical.  By the first condition
we mean that $u$, $v$, and $w$ are the strings of labels on or outside
the three border segments of $P$.  If $P$ is a flawed puzzle for $X$
that contains a marked scab $s$, then we set $\cA_T(P) = \cA_T(s)$.
Recall also that $\cA(P)$ is the aura of the right gash in the
resolution of $P$.

\begin{lemma}\label{lem:scabweight}
  If $P$ is any flawed puzzle for $X$ containing a marked scab $s$,
  then we have $\cA_T(P) = - \weight(s)\, \cA(P)$.
\end{lemma}
\begin{proof}
  Let $(i,j)$ be the position of $s$, and assume that the labels of
  $s$ and its resolution $\wtil s$ are as follows.
  \[
  {
    \psfrag{a}{\raisebox{.5mm}{$a$}}
    \psfrag{c}{\raisebox{.5mm}{$\,b$}}
    \psfrag{x}{\raisebox{.5mm}{$d$}}
    \psfrag{z}{\raisebox{.5mm}{$c$}}
    \psfrag{e}{\raisebox{.5mm}{$x$}}
    s = \raisebox{-9.5mm}{\pic{.8}{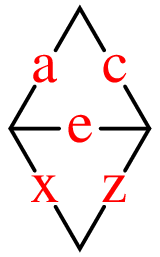}}
  }
  \ \ \ \ \ \ ; \ \ \ \ \ \ \ 
  {
    \psfrag{x}{\raisebox{.5mm}{$d$}}
    \psfrag{a}{\raisebox{.5mm}{$a$}}
    \psfrag{xc}{$b$}
    \psfrag{yc}{$c$}
    \wtil s = \raisebox{-10mm}{\pic{.8}{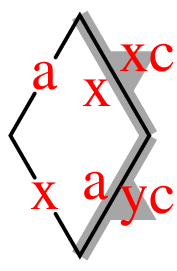}}
  }
  \]
  Since the gashes $a/c$ and $d\backslash b$ are in opposite classes
  by Proposition~\ref{prop:scabres}, we obtain $\cA_T(s) = \cA(c/a)
  y_j + \cA(b\backslash d) y_i = - \cA(a/c) (y_j-y_i) = - \cA(P)
  \weight(s)$.  The same calculation holds if the gashes are on the
  left side of the resolution of $s$.
\end{proof}

\begin{proof}[Proof of Theorem~\ref{thm:puzzle}]
  For each triple $(u,v,w)$ of 012-strings for $X$ we let $\wh
  C^w_{u,v} \in \La$ denote the equivariant class defined by the right
  hand side of Theorem~\ref{thm:puzzle}.  In other words we set $\wh
  C^w_{u,v} = \sum_P \weight(P)$ where the sum is over all equivariant
  puzzles for $X$ with boundary $\border^{u,v}_w$.  It follows from
  Proposition~\ref{prop:extreme} that these constants satisfy equation
  (\ref{eqn:extreme}).  We must show that they also satisfy equation
  (\ref{eqn:recurse}).

  Fix $u$, $v$, and $w$, and let $S$ be the set of all flawed puzzles
  for $X$ with boundary $\border^{u,v}_w$.  Since the mutation
  algorithm preserves the set of positions of equivariant pieces and
  marked scabs in a flawed puzzle, it follows from
  Proposition~\ref{prop:sum_flawed} and Lemma~\ref{lem:scabweight}
  that
  \begin{equation}\label{eqn:twosumzero}
    \sum_{P \in S_{\scab}} \cA_T(P) \weight(P) \ = \ 
    \sum_{P \in S_{\gash}} \cA(P) \weight(P) \,.
  \end{equation}
  Here the weight of a flawed puzzle is defined as the product of the
  weights of its equivariant pieces.  By rewriting the left hand side
  of (\ref{eqn:twosumzero}) as a sum over (flawless) equivariant
  puzzles for $X$ and applying Proposition~\ref{prop:sum_equiv_aura}
  we obtain
  \[
  \begin{split}
    \sum_{P\in S_\scab} \cA_T(P) \weight(P) \ 
    &= \ 
    \sum_{\partial P = \border^{u,v}_w} \weight(P) 
    \sum_{s \in \scabs(P)} \cA_T(s) \\
    &= \ 
    \sum_{\partial P = \border^{u,v}_w} \weight(P)\, (C_u \zeta^{11} + C_v
    \zeta^7 + C_w \zeta^3)  \\
    &= \ 
    (C_u \zeta^{11} + C_v \zeta^7 + C_w \zeta^3)\, \wh C^w_{u,v} \,.
  \end{split}
  \]
  
  Assume that $P$ is a puzzle in the second sum of
  (\ref{eqn:twosumzero}) with a gash-pair on the left border segment.
  If $u'$ is the string of labels on or inside this border segment,
  then we have $u \to u'$.  Furthermore, if $i$ is the smallest index
  for which $u_i \neq u'_i$, then $\cA(P) = \cA(u_i/u'_i) = \zeta^5\,
  \delta(\frac{u}{u'})$.  Similar identities hold for puzzles with
  gash pairs on the right or bottom border segments.  The second sum
  in (\ref{eqn:twosumzero}) can therefore be rewritten as:
  \[
  \sum_{P \in S_{\gash}} \cA(P) \weight(P) \ = \ \sum_{u \to u'}
  \zeta^5\, \delta(\frac{u}{u'})\, \wh C^w_{u',v} + \sum_{v \to v'}
  \zeta\, \delta(\frac{v}{v'})\, \wh C^w_{u,v'} + \sum_{w' \to w}
  \zeta^9\, \delta(\frac{w'}{w})\, \wh C^{w'}_{u,v} \,.
  \]
  
  We conclude that the identity (\ref{eqn:twosumzero}) is equivalent
  to equation (\ref{eqn:recurse}).  Since the constants $\wh
  C^w_{u,v}$ satisfy the identities (\ref{eqn:extreme}) and
  (\ref{eqn:recurse}), it follows from Theorem~\ref{thm:recurse} that
  they are the equivariant Schubert structure constants of $X$.  This
  completes the proof.
\end{proof}

%%% Local Variables: 
%%% mode: latex
%%% TeX-master: "puzzle2eq.tex"
%%% End: 

%% file: bibliography.tex
\providecommand{\bysame}{\leavevmode\hbox to3em{\hrulefill}\thinspace}
\providecommand{\MR}{\relax\ifhmode\unskip\space\fi MR }
% \MRhref is called by the amsart/book/proc definition of \MR.
\providecommand{\MRhref}[2]{%
  \href{http://www.ams.org/mathscinet-getitem?mr=#1}{#2}
}
\providecommand{\href}[2]{#2}

%%% Local Variables: 
%%% mode: latex
%%% TeX-master: "puzzle2eq.tex"
%%% End: 

%% file: puzzle2eq.bbl
\begin{thebibliography}{10}

\bibitem{andersen.jantzen.ea:representations}
H.~H. Andersen, J.~C. Jantzen, and W.~Soergel, \emph{Representations of quantum
  groups at a {$p$}th root of unity and of semisimple groups in characteristic
  {$p$}: independence of {$p$}}, Ast\'erisque (1994), no.~220. \MR{1272539
  (95j:20036)}

\bibitem{anderson.fulton:equivariant}
D.~Anderson and W.~Fulton, \emph{Equivariant cohomology in algebraic geometry},
book in preparation, 
{\tt http://people.math.osu.edu/anderson.2804/eilenberg/}.

\bibitem{beazley.bertiger.ea:rim-hook}
E.~Beazley, A.~Bertiger, and K.~Taipale, \emph{An equivariant rim hook
  rule for cohomology of {G}rassmannians}, preprint, 2013.

\bibitem{billey:kostant}
S.~Billey, \emph{Kostant polynomials and the cohomology ring for {$G/B$}},
  Proc. Nat. Acad. Sci. U.S.A. \textbf{94} (1997), no.~1, 29--32. \MR{1425869
  (98e:14051)}

\bibitem{buch:3step}
A.~S. Buch, \emph{On the puzzle conjecture for three-step flag manifolds}, in
  preparation.

\bibitem{buch:saturation}
\bysame, \emph{The saturation conjecture (after {A}.\ {K}nutson and {T}.\
  {T}ao)}, Enseign. Math. (2) \textbf{46} (2000), no.~1-2, 43--60, With an
  appendix by William Fulton. \MR{1769536 (2001g:05105)}

\bibitem{buch:quantum}
\bysame, \emph{Quantum cohomology of {G}rassmannians}, Compositio Math.
  \textbf{137} (2003), no.~2, 227--235. \MR{1985005 (2004c:14105)}

\bibitem{buch.kresch.ea:puzzle} A.~S. Buch, A.~Kresch, K.~Purbhoo, and
  H.~Tamvakis, \emph{The puzzle conjecture for the cohomology of
    two-step flag manifolds}, preprint, 2013.

\bibitem{buch.kresch.ea:gromov-witten}
A.~S. Buch, A.~Kresch, and H.~Tamvakis, \emph{Gromov-{W}itten invariants on
  {G}rassmannians}, J. Amer. Math. Soc. \textbf{16} (2003), no.~4, 901--915
  (electronic). \MR{1992829 (2004h:14060)}

\bibitem{buch.mihalcea:quantum}
A.~S. Buch and L.~Mihalcea, \emph{Quantum {$K$}-theory of {G}rassmannians},
  Duke Math. J. \textbf{156} (2011), no.~3, 501--538. \MR{2772069
  (2011m:14092)}

\bibitem{buch.mihalcea:curve}
\bysame, \emph{Curve neighborhoods of {S}chubert varieties}, preprint, 2013.

\bibitem{chevalley:decompositions}
C.~Chevalley, \emph{Sur les d\'ecompositions cellulaires des espaces {$G/B$}},
  Proc. Sympos. Pure Math., vol.~56, 1994, pp.~1--23. \MR{1278698 (95e:14041)}

\bibitem{coskun:littlewood-richardson}
I.~Co\c{s}kun, \emph{A {L}ittlewood-{R}ichardson rule for two-step flag varieties},
  Invent. Math. \textbf{176} (2009), no.~2, 325--395. \MR{2495766
  (2010e:14048)}

\bibitem{coskun.vakil:geometric}
I.~Co\c{s}kun and R.~Vakil, \emph{Geometric positivity in the cohomology of
  homogeneous spaces and generalized {S}chubert calculus},
  Proc. Sympos. Pure Math., vol.~80, 2009, pp.~77--124.
  \MR{2483933 (2010d:14074)}

\bibitem{fulton:young}
W.~Fulton, \emph{Young tableaux}, London Mathematical Society Student Texts,
  vol.~35, 1997. \MR{1464693 (99f:05119)}

\bibitem{fulton:intersection}
\bysame, \emph{Intersection theory}, second ed., Ergebnisse der Mathematik und
  ihrer Grenzgebiete (3), vol.~2, Springer-Verlag, Berlin, 1998. \MR{1644323
  (99d:14003)}

\bibitem{fulton.pandharipande:notes}
W.~Fulton and R.~Pandharipande, \emph{Notes on stable maps and quantum
  cohomology}, Proc. Sympos. Pure Math., vol.~62, 1997, pp.~45--96.
  \MR{1492534 (98m:14025)}

\bibitem{graham:positivity}
W.~Graham, \emph{Positivity in equivariant {S}chubert calculus}, Duke Math. J.
  \textbf{109} (2001), no.~3, 599--614. \MR{1853356 (2002h:14083)}

\bibitem{kim:equivariant}
B.~Kim, \emph{On equivariant quantum cohomology}, Internat. Math. Res. Notices
  (1996), no.~17, 841--851. \MR{1420551 (98h:14013)}

\bibitem{knutson:puzzles}
A.~Knutson, \emph{Puzzles, {P}ositroid varieties, and equivariant {$K$}-theory
  of {G}rassmannians}, arXiv:1008.4302.

\bibitem{knutson:conjectural}
\bysame, \emph{A conjectural rule for ${GL}_n$ {S}chubert calculus},
  unpublished manuscript, 1999.

\bibitem{knutson.purbhoo:product}
A.~Knutson and K.~Purbhoo, \emph{Product and puzzle formulae for {${\rm GL}_n$}
  {B}elkale-{K}umar coefficients}, Electron. J. Combin. \textbf{18} (2011),
  no.~1, Paper 76, 20. \MR{2788693 (2012f:14095)}

\bibitem{knutson.tao:puzzles}
A.~Knutson and T.~Tao, \emph{Puzzles and (equivariant) cohomology of
  {G}rassmannians}, Duke Math. J. \textbf{119} (2003), no.~2, 221--260.
  \MR{1997946 (2006a:14088)}

\bibitem{knutson.tao.ea:honeycomb}
A.~Knutson, T.~Tao, and C.~Woodward, \emph{The honeycomb model of {${\rm
  GL}_n(\mathbb C)$} tensor products. {II}. {P}uzzles determine facets of the
  {L}ittlewood-{R}ichardson cone}, J. Amer. Math. Soc. \textbf{17} (2004),
  no.~1, 19--48. \MR{2015329 (2005f:14105)}

\bibitem{kontsevich.manin:gromov-witten}
M.~Kontsevich and Yu. Manin, \emph{Gromov-{W}itten classes, quantum cohomology,
  and enumerative geometry}, Comm. Math. Phys. \textbf{164} (1994), no.~3,
  525--562. \MR{1291244 (95i:14049)}

\bibitem{kostant.kumar:nil*1}
B.~Kostant and S.~Kumar, \emph{The nil {H}ecke ring and cohomology of {$G/P$}
  for a {K}ac-{M}oody group {$G$}}, Adv. in Math. \textbf{62} (1986), no.~3,
  187--237. \MR{866159 (88b:17025b)}

\bibitem{kreiman:equivariant}
V.~Kreiman, \emph{Equivariant {L}ittlewood-{R}ichardson skew tableaux}, Trans.
  Amer. Math. Soc. \textbf{362} (2010), no.~5, 2589--2617. \MR{2584612
  (2011d:05385)}

\bibitem{mihalcea:equivariant}
L.~Mihalcea, \emph{Equivariant quantum {S}chubert calculus}, Adv. Math.
  \textbf{203} (2006), no.~1, 1--33. \MR{2231042 (2007c:14061)}

\bibitem{mihalcea:positivity}
\bysame, \emph{Positivity in equivariant quantum {S}chubert calculus}, Amer. J.
  Math. \textbf{128} (2006), no.~3, 787--803. \MR{2230925 (2007c:14062)}

\bibitem{mihalcea:equivariant*1}
\bysame, \emph{On equivariant quantum cohomology of homogeneous spaces:
  {C}hevalley formulae and algorithms}, Duke Math. J. \textbf{140} (2007),
  no.~2, 321--350. \MR{2359822 (2008j:14106)}

\bibitem{molev:littlewood-richardson*1}
A.~Molev, \emph{Littlewood-{R}ichardson polynomials}, J. Algebra \textbf{321}
  (2009), no.~11, 3450--3468. \MR{2510056 (2010e:05316)}

\bibitem{molev.sagan:littlewood-richardson}
A.~Molev and B.~Sagan, \emph{A {L}ittlewood-{R}ichardson rule for factorial
  {S}chur functions}, Trans. Amer. Math. Soc. \textbf{351} (1999), no.~11,
  4429--4443. \MR{1621694 (2000a:05212)}

\bibitem{monk:geometry}
D.~Monk, \emph{The geometry of flag manifolds}, Proc. London Math. Soc. (3)
  \textbf{9} (1959), 253--286. \MR{0106911 (21 \#5641)}

\bibitem{ruan.tian:mathematical}
Y.~Ruan and G.~Tian, \emph{A mathematical theory of quantum cohomology}, Math.
  Res. Lett. \textbf{1} (1994), no.~2, 269--278. \MR{1266766 (95b:58025)}

\end{thebibliography}
